\documentclass[11pt,a4paper]{article}

\usepackage{epsf,epsfig,amsfonts,amsgen,amsmath,amstext,amsbsy,amsopn,amsthm
}
\usepackage{enumerate}
\usepackage{amsmath}
\usepackage{amsfonts,amsthm,amssymb}
\usepackage{amsfonts}
\usepackage{graphics}
\usepackage{latexsym,bm}
\usepackage{amsfonts,amsthm,amssymb,bbding}
\usepackage{indentfirst}
\usepackage{graphicx}
\usepackage{color}
\usepackage[colorlinks=true,anchorcolor=blue,filecolor=blue,linkcolor=blue,urlcolor=blue,citecolor=red]{hyperref}
\usepackage{moresize}
\usepackage{array}
\usepackage{multirow}
\usepackage{graphicx}
\usepackage{threeparttable}
\usepackage{makecell}
\renewcommand\proofname{\bf Proof}
\newtheorem{theorem}{Theorem}

\newtheorem{lemma}{Lemma}[section]

\newtheorem{false statement}{False statement}

\newtheorem{fact}{Fact}

\theoremstyle{definition}

\newtheorem{claim}{Claim}

\newtheorem{conjecture}{Conjecture}
\newtheorem{corollary}{Corollary}[section]

\begin{document}

\title{\textbf{Minimizing the spectral radius of unicyclic graphs with fixed order and dissociation number}}
\author{\Large Yu Zhou$^{a}$, Yuanyuan Chen$^{a}$\thanks{Corresponding author;
Email addresses:
chenyy\underline{\ }de@sina.com (Y. Chen), zyymee@163.com (Y. Zhou), ldxjedu@163.com (D. Li),
ddfan0526@163.com (D. Fan).}, Dan Li$^{a},$ 
Dandan Fan$^{b}$ \\
\small\it $^a$ College of Mathematics and Systems Science, Xinjiang University,\\
\small\it  Urumqi, Xinjiang 830017, China\\[1mm]
\small\it $^b$ College of Mathematics and Physics, Xinjiang Agricultural University,\\
\small\it Urumqi, Xinjiang 830052, China\\[1mm]}
\date{}

\maketitle

\begin{abstract}

A dissociation set of a graph $G$ is a vertex subset $S\subseteq V(G)$ such that the induced subgraph $G[S]$ has maximum degree  less than or equal to 1, and the dissociation number of $G$ is the cardinality of a maximum dissociation set denoted by $\varphi(G)$. In this paper, we first derive a sharp lower bound on the dissociation number for all unicyclic graphs. Furthermore, we also study the graphs with minimum spectral radius  among all unicyclic graphs with given order and dissociation number and  explicitly identify all extremal graphs for $n-4\leq \varphi(G)\leq n-1$. 
\end{abstract}


\section{Introduction}

Throughout this paper, we consider only simple, finite, connected graphs. Let $G$ be a graph with vertex set $V(G)$ and edge set $E(G)$. Suppose that $N_G(v)$ is the set of neighbors of a vertex $v \in V(G)$, and $d_G(v) = |N_G(v)|$. A vertex $v$ with $d_G(v)=1$ is called a pendant vertex, and its unique incident edge is a pendant edge. For any subset $S\subseteq V(G)$, we denote by \( G[S] \) the subgraph of \( G \) induced by \( S \), and  by \( G - S \) the subgraph induced by \( G[V(G)\setminus S]\). 
A subset \( I \subseteq V(G) \)  is called an independent set if no two vertices in $I$ are adjacent in $G$. The cardinality of a maximum independent set of $G$ is the independence number. A subset \( S \subseteq V(G) \)  is a \textbf{dissociation set} if \( G[S] \) has maximum degree at most 1. The \textbf{dissociation number} \( \varphi(G) \) is the size of a maximum dissociation set in \( G \). In particular, if the maximum degree of $G[S]$ is zero, then this dissociation set is an independent set. Bock et al. studied the relationship between the independence number and the dissociation number in \cite{bock2022, Bock2023} and investigated the bounds of the dissociation number in \cite{Bock2023a}. Orlovich et al. \cite{Orlovich2011} systematically investigated the complexity of dissociation set problems in graphs. Li, Tu et al.  \cite{li2026, {Tu2024},Tu2021} counted the maximum and second maximum numbers of dissociation sets. Note that a unicyclic graph is a connected graph whose number of edges equals its number of vertices. Zhou, Liu and Xiong \cite{zhao2025} showed \( \varphi(G) \geq \left\lceil \frac{2n}{3} \right\rceil \) for an  \( n \)-order unicyclic graph $G$ satisfying \( 3 \mid n \). We derive a general lower bound for the dissociation number of unicyclic graphs in this work.

\begin{theorem}\label{minU}
Let \( G \) be a unicyclic graph of order \( n \) with dissociation number \( \varphi(G) \). Then
\(
\left\lfloor \frac{2n}{3} \right\rfloor \le \varphi(G) \le n-1.
\)
\end{theorem}

The adjacency matrix of \( G \) is \( A(G) = (a_{ij}) \), where \( a_{ij} = 1 \), if \( v_i \) and \( v_j \) are adjacent, and \( a_{ij} = 0 \) otherwise. The eigenvalues of $G$ are the eigenvalues of its adjacency matrix $A(G)$. The largest eigenvalue of \( A(G) \) is called the spectral radius of \( G \) and denoted by \( \rho(G) \).  
The spectral radius serves as a core research object in spectral graph theory. A fundamental problem in extremal spectral theory is to characterize the graphs that achieve the maximum or minimum spectral radius over a given family of graphs, which is widely recognized as the Brualdi-Solheid problem. The spectral radius of the adjacency matrix determines the threshold of epidemic and information propagation on networks, where a smaller spectral radius effectively restrains diffusion behaviors. Ai, Kim et al. \cite{ai2025minimum,Kim2020,Kim2023} considered the minimal spectral radius of saturated graphs. We link graphs attaining the minimum spectral radius to the independence number and the dissociation number. 
There has been extensive research concerning the independence number.
We refer the reader to \cite{Choi2024,Du2013,Hu2025,Lou2022,Xu2009,xue2024}. As a generalization of the independence number, relatively few results regarding the dissociation number and its connections to graph invariants have been reported in the literature. Das and Mohanty \cite{das2024maximization} investigated the maximal spectral radius of block graphs with a prescribed dissociation number. 
In \cite{huang2026, Huang2024, zhao2025}, Huang, Zhao et al.  determined the graphs with the minimum spectral radius among all connected graphs with given order $n$ and dissociation number in $\{2, \left\lceil \frac{2n}{3} \right\rceil - 1, \left\lceil \frac{2n}{3} \right\rceil, n-3, n-2, n-1\}$. Focusing on unicyclic graphs, Cui et al. \cite{Cui2026} analyzed the graphs with the minimum spectral radius in unicyclic graphs with  domination number and independence number. Fan and Wu \cite{Fan2008} ranked unicyclic graphs by the spectral radius of unoriented Laplacian matrices.  Li, Wang and Zhao \cite{Li2011} characterized the signless Laplacian spectral radius of unicyclic and bicyclic graphs with prescribed girth. Zhang, Ren and Wang \cite{Zhang2026} investigated unicyclic graphs attaining the minimum number of maximal dissociation sets. In this paper, we  characterize the extremal graphs with given dissociation number $\varphi(G)$ that attain the minimal spectral radius among all unicyclic graphs. Furthermore, the
minimizing graphs are explicitly determined for $n-4 \leq \varphi(G) \leq n-1$.

We denote by $P_n$ the path graph on $n$ vertices.
The \emph{length} of $P_n$ is the number of edges in the path. Accordingly, a path of length $l$ (with $l$ edges) has $l+1$ vertices and is written as $P_{l+1}$. Let \(X = \{ w_1, w_2, \dots, w_k \}\subseteq V(G) \), for a sequence \( p_{\le l}(X)\!=\!p_{\le l}(p(w_1),p(w_2),\!\dots\!,p(w_k)) \) with integers \(l\ge 1,\, p(w_i) \ge 0 \). The graph \( \mathcal{G}^* = G \bullet p_{\le l}(X) \) is obtained by attaching \( p(w_i) \) disjoint paths of length at most $l$ (i.e., $P_{\le l+1}$ ) to each vertex \( w_i \) in \( X \). Every path is attached by identifying exactly one of its endpoints with
$w_i$. In particular,  \( \mathcal{G}^* = G \bullet p_{l}(X) \) is obtained by attaching \( p(w_i) \) disjoint paths of length equal to $l$ (i.e., $P_{ l+1}$ ) to each vertex \( w_i \) in \( X \). If \( p(w_{i_1})  = \dots = p(w_{i_t}) =p(w_1)>p(w_{i_{t+1}}) = \dots = p(w_{i_k})=p(w_k) \) with \( 0 \le t \le k \), then we write \( \mathcal{G}^* = G \bullet p_l \left( (p(w_k))^{[k-t]}, (p(w_1))^{[t]} \right) \). 
The 2-subdivision graph \( S_2(G) \) of \( G \) is obtained by subdividing each edge of \( G \) twice with two new vertices, and these added vertices are called subdivision vertices.  Let   $X=V(C_k)$ and \( \left\lfloor \frac{n-3k}{2} \right\rfloor \equiv r \pmod{k} \) with integer $0\le r \le k-1$. If 
\(\varphi\) is even, let $\mathcal{G}_1^* = S_2(C_k) \bullet p_2\left(
\left( \frac{n-2r-k}{2k} - 1 \right)^{[k-r]},\right.
\left. \left( \frac{n-2r-k}{2k} \right)^{[r]}
\right)$, see Figure 1 (a).
If $\varphi$ is odd, let \(\mathcal{G}^*_2=S_2(C_k) \bullet p_2\left( \left( \frac{n-2r-k-1}{2k} - 1 \right)_1^{{[k-r]}}  ,\left( \frac{n-2r-k-1}{2k} \right)^{{[r]}}\right)\), this graph is constructed by attaching one pendant  vertex (denoted $u$) by a single edge to an arbitrary vertex $w_{ik}\in V(C_k)$ with fewer incident pendant paths,  
see Figure 1 (b).
\begin{figure}
    \centering
\includegraphics[width=0.4\linewidth]{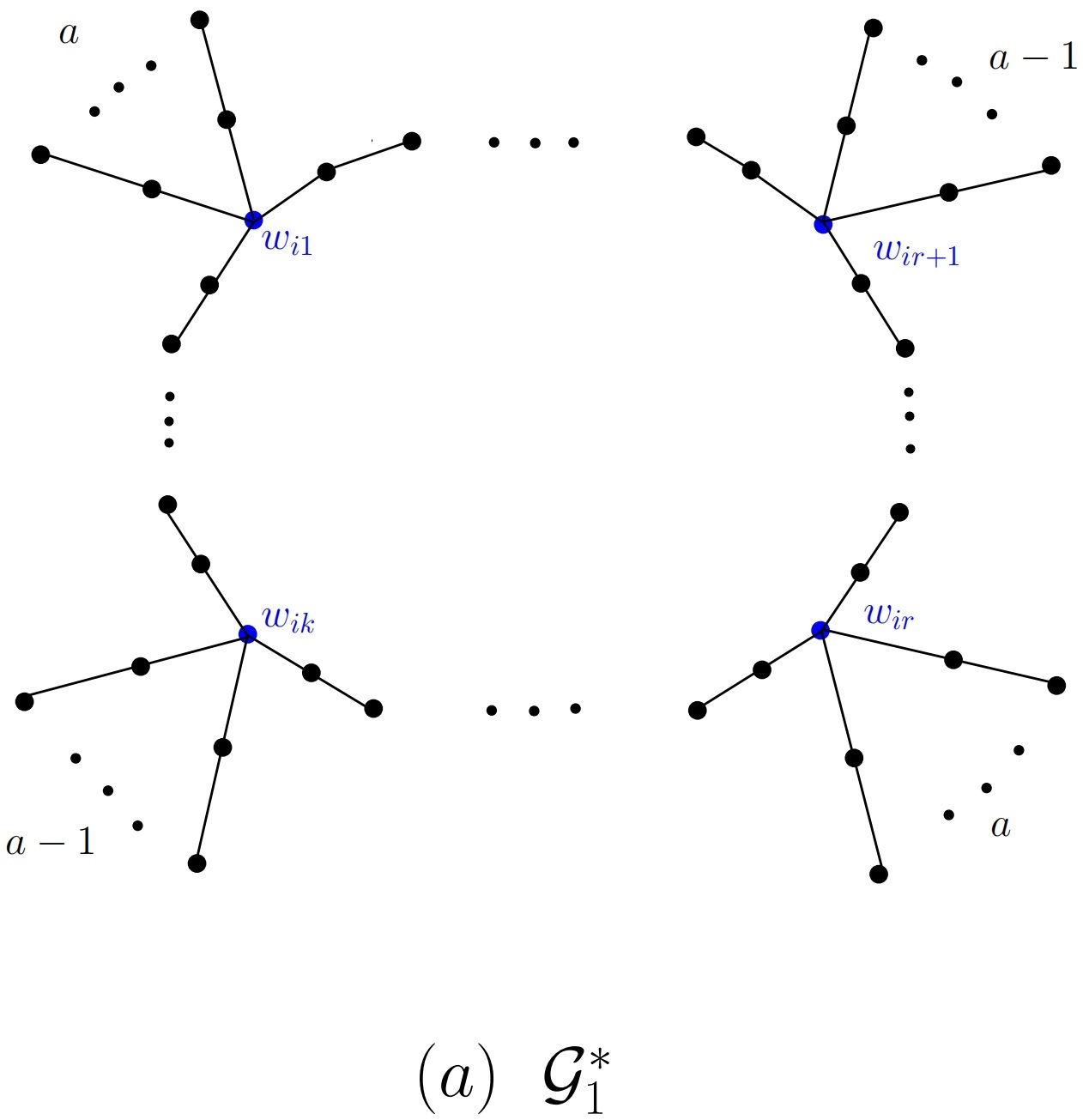}
\includegraphics[width=0.4\linewidth]{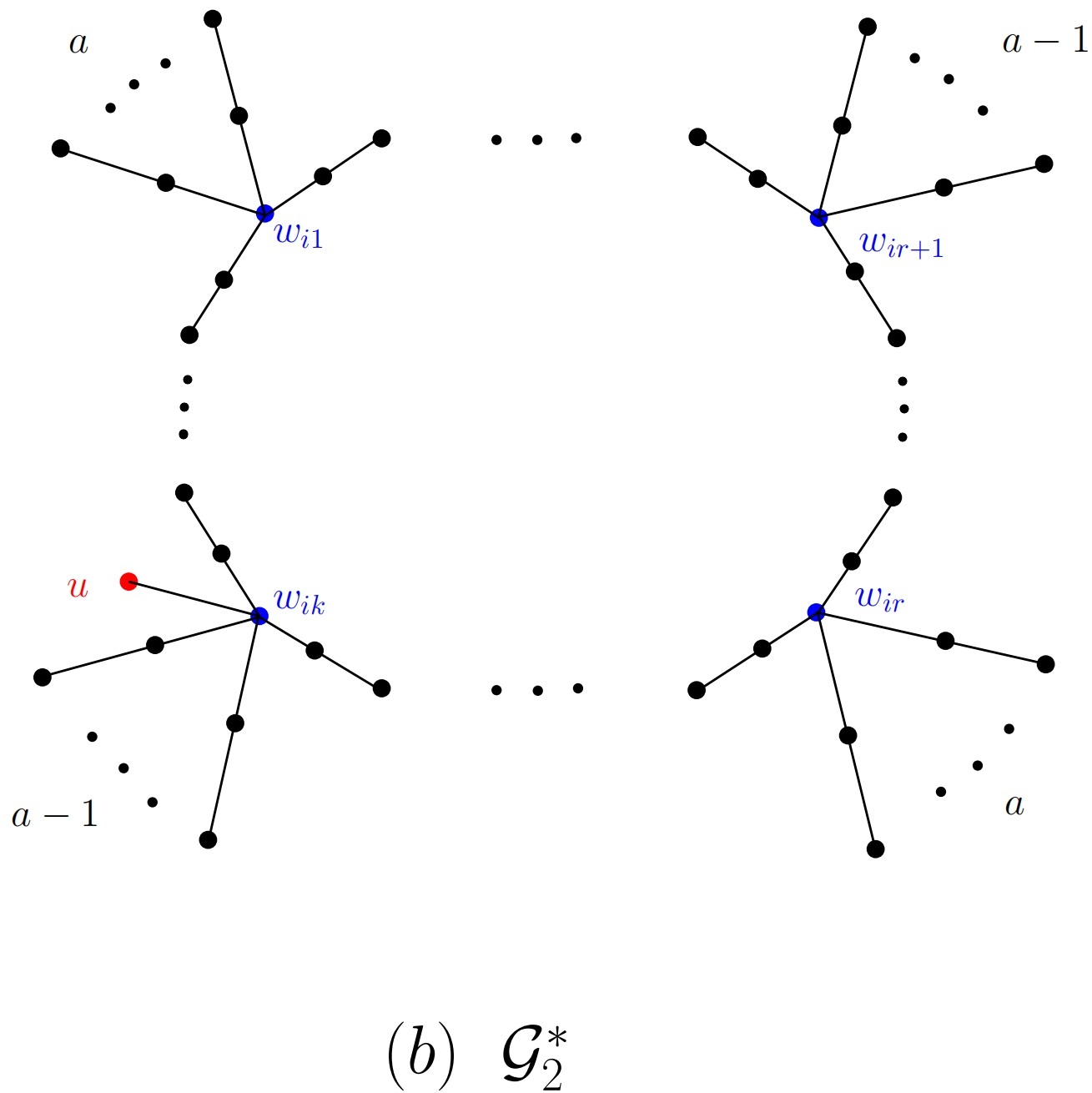}
        \caption{ graphs $\mathcal{G}_1^* \cup \mathcal{G}_2^*$, where~\(a=\left\lfloor\frac{n-2r-k}{2k} \right\rfloor\).}
\end{figure}

For integers $n,\varphi$, let $\mathcal{U}_{n,\varphi}$ be the set of unicyclic graphs with order $n$ and dissociation number $\varphi$. For any two vertices $u,v\in V(G)$, the distance $\mathrm{dist}(u,v)$ between $u$ and $v$
is defined as the number of edges on a shortest path connecting $u$ and $v$.
The following theorems characterize the structure of the minimizer graphs in the class
$\mathcal{U}_{n,\varphi}$ for $ \varphi(G) \leq n-1$.

\begin{theorem}\label{T2}
Suppose that \( G^* \) is the minimizer graph in \( \mathcal{U}_{n,\varphi} \).
Then we have

(i). When $\varphi = n-1$, let $X=\{w_i\}\subset V(C_3)$. Then 
\[
G^* =
\begin{cases}
C_3\bullet p_2\big(\frac{n-3}{2}\big), & \varphi \text{ is even};\\[5pt]
C_3\bullet p_2\big(\left\lfloor \frac{n-3}{2}\right\rfloor_1\big), & \varphi \text{ is odd}.
\end{cases}
\]

(ii). When $\varphi = n-2$, let $X=\{w_i, w_j\}\subset V(C_6)$ such that $dist(w_i, w_j)=3$. Then  
\[
G^* =
\begin{cases}
C_6\bullet p_2\big(\left\lfloor \frac{n-6}{4}\right\rfloor,\left\lceil \frac{n-6}{4}\right\rceil\big), & \varphi \text{ is even};\\[5pt]
C_6\bullet p_2\big(\left\lfloor \frac{n-6}{4}\right\rfloor_1,\left\lceil \frac{n-6}{4}\right\rceil\big), & \varphi \text{ is odd}.
\end{cases}
\]
\end{theorem}

\begin{theorem}\label{T3}
Suppose that \( G^* \) is the minimizer graph in \( \mathcal{U}_{n,\varphi} \), where  $\varphi \le n-3.$
Let \( k = n - \varphi \), \( \left\lfloor \frac{n-3k}{2} \right\rfloor \equiv r \pmod{k} \) and 
let $X=V(C_k)$. Then
\[
G^* \in
\begin{cases}
S_2(C_k) \bullet p_2\left( \left( \frac{n-2r-k}{2k} - 1 \right)^{{[k-r]}} ,\left( \frac{n-2r-k}{2k} \right)^{{[r]}}\right),  &\text{ when}~ {\varphi} \text{ is even};\\[5pt]
S_2(C_k) \bullet p_2\left( \left( \frac{n-2r-k-1}{2k} - 1 \right)^{{[k-r]}}_{1}  ,\left( \frac{n-2r-k-1}{2k} \right)^{{[r]}}\right), &\text{ when}~{\varphi} \text{ is odd.}
\end{cases}
\]

\end{theorem}

From Theorem \ref{T3}, we obtain the following result immediately.

\begin{theorem}\label{T4}
Suppose that \(G^* \) is the minimizer graph in \( \mathcal{U}_{n, n-3} \) and \( \left\lfloor \frac{n-9}{2} \right\rfloor \equiv r \pmod{3} \), where $0 \leq r \leq 2$. Let $X=V(C_3)$. Then we have the following statements.

(i). If \( n-3 \) is even, then
\[
G^* =
\begin{cases}
S_2(C_3) \bullet p_2\left( \left( \frac{n-9}{6} \right)^{[3]} \right), & \text{when } r = 0; \\
S_2(C_3) \bullet p_2\left( \left( \frac{n-11}{6} \right)^{[2]} , \left( \frac{n-5}{6} \right)^{[1]} \right), & \text{when } r = 1; \\
S_2(C_3) \bullet p_2\left( \left( \frac{n-13}{6} \right)^{[1]} , \left( \frac{n-7}{6} \right)^{[2]} \right), & \text{when } r = 2.
\end{cases}
\]

(ii). If \( n-3\) is odd, then
\[
G^* =
\begin{cases}
S_2(C_3) \bullet p_2\left( \left( \frac{n-10}{6} \right)^{[3]}_1 \right), & \text{when } r = 0; \\
S_2(C_3) \bullet p_2\left( \left( \frac{n-12}{6} \right)^{[2]}_1 , \left( \frac{n-6}{6} \right)^{[1]} \right), & \text{when } r = 1; \\
S_2(C_3) \bullet p_2\left( \left( \frac{n-14}{6} \right)^{[1]}_1 , \left( \frac{n-8}{6} \right)^{[2]} \right), & \text{when } r = 2.
\end{cases}
\]
\end{theorem}

\vskip 0.2in
Suppose that \( \left\lfloor \frac{n-2r-k}{2k} \right\rfloor\) disjoint pendant  $P_3$-paths are attached to \( w_{i_1}, \dots, w_{i_{r}} \) in $G$,  respectively. If
$\operatorname{dist}(w_{i_s}, w_{i_{s}+1}) \in \left\{ 3\left\lceil \frac{k}{r} \right\rceil, 3\left\lfloor \frac{k}{r} \right\rfloor \right\}$
for \( 1 \leq s \leq r - 1 \), then
\[
G =
\begin{cases}
S_2(C_k) \bullet \mathfrak{b} p_2\left( \left( \frac{n-2r-k}{2k} - 1 \right)^{{[k-r]}} ,\left( \frac{n-2r-k}{2k} \right)^{{[r]}}\right), &\text{when } {(  n-k)} \text{ is even;} \\
S_2(C_k) \bullet \mathfrak{b} p_2\left( \left( \frac{n-2r-k-1}{2k} - 1 \right)^{{[k-r]}}_{1}  ,\left( \frac{n-2r-k-1}{2k} \right)^{{[r]}}\right),&\text{when } {(  n-k)} \text{ is odd.}
\end{cases}
\]
\vskip 0.2in

Using Theorem \ref{T3}, we can precisely characterize the minimizing graph in $\mathcal{U}_{n,n-4}$ as follows.
\begin{theorem}\label{P2}
Suppose that \( G^* \in \mathcal{U}_{n, n-4} \) is the minimizer graph. 
Let $X=V(C_4)$ and \( \left\lfloor \frac{n-12}{2} \right\rfloor \equiv r \pmod{4} \), where $0 \leq r \leq 3$. Then we have the following statements.

(i). If \( n-4 \) is even, then

\[
G^* =
\begin{cases}
S_2(C_4) \bullet p_2\left( \left( \frac{n-12}{8} \right)^{[4]} \right), &\text{when } r = 0; \\[8pt]
S_2(C_4) \bullet p_2\left( \left( \frac{n-14}{8} \right)^{[3]}, \left( \frac{n-6}{8} \right)^{[1]} \right), & \text{when }r = 1; \\[8pt]
S_2(C_4) \bullet \mathfrak{b} p_2\left( \left( \frac{n-16}{8} \right)^{[2]}, \left( \frac{n-8}{8} \right)^{[2]} \right), & \text{when }r = 2; \\[8pt]
S_2(C_4) \bullet p_2\left( \left( \frac{n-18}{8} \right)^{[1]}, \left( \frac{n-10}{8} \right)^{[3]} \right), &\text{when } r = 3.
\end{cases}
\]
\vskip 0.15in
(ii). If \( n-4 \) is odd, then 
\[
G^* =
\begin{cases}
S_2(C_4) \bullet p_2\left( \left( \frac{n-13}{8} \right)^{[4]}_1 \right), &\text{when } r = 0; \\[8pt]
S_2(C_4) \bullet \mathfrak{b} p_2\left( \left( \frac{n-15}{8} \right)^{[3]}_1, \left( \frac{n-7}{8} \right)^{[1]} \right), & \text{when }r = 1; \\[8pt]
S_2(C_4) \bullet \mathfrak{b} p_2\left( \left( \frac{n-17}{8} \right)^{[2]}_1, \left( \frac{n-9}{8} \right)^{[2]} \right), & \text{when }r = 2; \\[8pt]
S_2(C_4) \bullet p_2\left( \left( \frac{n-19}{8} \right)^{[1]}_1, \left( \frac{n-11}{8} \right)^{[3]} \right), & \text{when }r = 3.
\end{cases}
\]
\end{theorem}

\section{ The lower bound on the dissociation number of unicyclic graphs}

For a connected graph \( G \), let \( P(G) \) denote the set of all pendant vertices of \( G \). Let \( Q(G) \) denote the set of all quasi-pendant vertices (vertices adjacent to pendant vertices). In particular, let \( Q_2(G) \) denote the set of all quasi-pendant vertices of degree exactly 2.

\begin{lemma}[\cite{Huang2024}]\label{LS(G)} 
Let \( G \) be a connected graph with \( n \geq 5 \) vertices and \( S^*(G) = P(G) \cup Q_2(G) \). Then there exists a maximum dissociation set \( S(G) \) such that \( S^*(G) \subseteq S(G) \).
\end{lemma}

For convenience, if \( P(G) \cup Q_2(G) \subseteq S(G) \), then \( S(G) \) is called the optimal dissociation set of \( G \).
Let $G$ be a graph and let $t$ be a positive integer. A subset of vertices $S \subseteq V(G)$ is called a $t$-path vertex cover if every path of order $t$ in $G$ contains at least one vertex from $S$. We denote by $\psi_t(G)$ the minimum cardinality of a $t$-path vertex cover in $G$.

\begin{lemma}[\cite{Bre2011a}]\label{T}
Let $T$ be a tree and $t$ be a positive integer. Then, $\psi_t(T) \le \frac{|V(T)|}{t}$.
\end{lemma}

\begin{lemma}[\cite{zhao2025}]\label{U}
If \( G \) is a unicyclic graph of order \( n \) and \( 3 \mid n \), then \( \varphi(G) \geq \left\lceil \frac{2n}{3} \right\rceil \).
\end{lemma}

Now, we give the proof of Theorem \ref{minU}. 

\renewcommand\proofname{\bf Proof of Theorem \ref{minU}}

\begin{proof}
Let G be a unicyclic graph of order $n$ with dissociation number $\varphi(G)$. Then $\varphi(G)\le n-1$. In what follows, we prove that $\varphi(G)\ge\lfloor \frac{2n}{3} \rfloor$. 
If \( 3 \mid n \), by Lemma \ref{U}, then \( \varphi(G) \geq \left\lceil \frac{2n}{3} \right\rceil\ge\lfloor \frac{2n}{3} \rfloor \) holds. Next, we suppose that $3 \nmid n$.
If \( G \cong C_n \), so the result holds. Otherwise,
let  \( C_q \) be the unique cycle of \( {G} \)
 with  $\ 3 \leq q \leq n-1 $. We proceed by induction on $n$. For $n=4$ and $G\cong C_3\bullet p_1\big(1\big)$, we get $\varphi(G)=3>\lfloor 8/3 \rfloor=2$, so the base case holds. Assume the result holds for all unicyclic graphs of order less than $n$. Let $v$ be a pendant vertex of $G$ with neighbor $u$.
If \( u \notin V(C_q) \) (where \( C_q \) is the unique cycle), let \( G_1 = G - \{v\} \). By the induction hypothesis, \( \varphi(G_1) \geq \lfloor \frac{2(n-1)}{3} \rfloor \). By Lemma \ref{LS(G)}, \( \varphi(G) \geq \varphi(G_1) + 1 \geq \lfloor \frac{2(n-1)}{3} \rfloor + 1 = \lfloor \frac{2n+1}{3} \rfloor \geq \lfloor \frac{2n}{3} \rfloor \).
If \( u \in V(C_q) \), then \( d_G(u) \geq 3 \). Let \( G_2 = G - \{u,v\} \). Then $G_2$ is a tree. For $t=3$, $\psi_3(G_2) = |V(G_2)| - \varphi(G_2)$.
Then by Lemma \ref{T}, we have \( \varphi(G_2) \geq  \frac{2(n-2)}{3}\).  By Lemma \ref{LS(G)} and $3 \nmid n$, we obtain \(
\varphi(G) \geq \varphi(G_2) + 1 \geq  \frac{2n-4}{3}  + 1 = \frac{2n-1}{3} \geq \lfloor \frac{2n}{3} \rfloor.
\)
\end{proof}

\section{Preliminaries for spectral minimizing graphs in \(\mathcal{U}_{n,\varphi} \)}

\begin{lemma}[\cite{Cvetkovic2009}]\label{LG(s,t)} 
Let \( G(s,t) \) be a graph obtained from a nontrivial connected graph by attaching two pendant paths of lengths \( s \) and \( t \) at a vertex \( w \in V(G) \). If \( s \geq t\geq1 \), then $\rho(G_{s,t}) > \rho(G_{s+1, t-1})$.
\end{lemma}

\begin{lemma}[\cite{Li1979}]\label{G}
Let \( G_1, G_2 \) be two connected graphs. Then we have the following statements.
\begin{enumerate}
\item[(i)] If \( G_2 \) is a proper subgraph of \( G_1 \), then $\rho(G_2) < \rho(G_1).$
\item[(ii)] If \( P(G_2, \lambda) > P(G_1, \lambda) \) for all \( \lambda \geq \rho(G_1) \), then $\rho(G_2) < \rho(G_1).$
\end{enumerate}
\end{lemma}

Let \( G \) be a connected graph. By the Perron-Frobenius theorem, \( A(G) \) has a unique positive unit eigenvector corresponding to \( \rho(G) \), which is called the {Perron vector} of \( A(G) \).

\begin{lemma}[\cite{Wu2005}]\label{N(u)}
Let $u, v$ be two vertices of the connected graph $G$. Suppose $v_{1}, v_{2}, \dots, v_{s}$ ($1 \leq s \leq d_{v}$) are some vertices of $N_{G}(v) \setminus N_{G}(u)$ and $\boldsymbol{x}=(x_{1}, x_{2}, \dots, x_{n})^{T}$ is the Perron vector of $G$, where $x_{i}$ corresponds to the vertex $v_{i}$ ($1 \leq i \leq n$).
 Let $ G^* = G - \{vv_i \mid 1 \leq i \leq s\} + \{uv_i \mid 1 \leq i \leq s\}$.
If $x_{u} \geq x_{v}$, then $\rho(G)<\rho(G^{*})$.
\end{lemma}

An {internal path} in a graph \( G \) is a sequence of vertices
\( w_1, w_2, \dots, w_k\) (\( k \geq 2 \)) such that: (i) the vertices in the sequence are distinct (except possibly \( w_1 = w_k \)); (ii) \( w_i \) is adjacent to \( w_{i+1} \) for \( i = 0,1,\dots,k-1 \); (iii) the vertex degrees satisfy \( d(w_1) \geq 3 \), \( d(w_2) = \cdots = d(w_{k-1}) = 2 \) (unless \( k \geq 2 \)), and \( d(w_k) \geq 3 \). Let \( W_n \) be the graph obtained from \( P_{n-2} \) by attaching one pendant edge to each quasi-pendant vertex of \( P_{n-2} \).

\begin{lemma}[\cite{Hoffman1975}]\label{G_w}
Suppose that  $G\ncong W_n$, and  $ e = uv $ is an edge on an interal path of $ G$. Let $ G_w $ be the graph obtained from  $G$  by once subdivision of the edge $e $. Then $\rho(G) > \rho(G_w).$
\end{lemma}

Let \( M \) be an \( n \times n \) real matrix, and let \( X = \{1, 2, \dots, n\} \). Consider a partition \( \Pi: X = X_1 \cup X_2 \cup \cdots \cup X_k \). The matrix \( M \) can be written according to this partition as:
\[
M = \begin{bmatrix}
M_{1,1} & M_{1,2} & \cdots & M_{1,k} \\
M_{2,1} & M_{2,2} & \cdots & M_{2,k} \\
\vdots & \vdots & \ddots & \vdots \\
M_{k,1} & M_{k,2} & \cdots & M_{k,k}
\end{bmatrix}.
\]
Then the {quotient matrix} of \( M \) with respect to \( \Pi \) is defined as the \( k\times k\) matrix \(B_{\Pi}=(b_{i,j})_{i,j=1}^k \), where \( b_{ij} \) is the average of the row sums of \( M_{i,j} \). If each block \( M_{i,j} \) has constant row sum, the partition \( \Pi \) is called \textbf{equitable}.

\begin{lemma}(Brouwer and Haemers \cite{Brouwer2011}, p. 30; Godsil and Royle \cite{godsil2001}, pp. 196--198)\label{B}
Let $M$ be a real symmetric matrix, and let $\lambda_1(M)$ be the largest eigenvalue of $M$. If $B_{\Pi}$ is a quotient matrix of $M$, then the eigenvalues of $B_{\Pi}$ interlace those of $M$. Furthermore, if the partition $\Pi$ is equitable, then all eigenvalues of $B_{\Pi}$ are also eigenvalues of $M$, and for any eigenvalue of $B_{\Pi}$ the corresponding eigenvector has constant element on each vertex in a common part.
\end{lemma}

A cut edge of a graph is an edge whose removal increases the number of connected components. For a graph \( G \) with \( n \) vertices, the characteristic polynomial of \( G \) is denoted by $P(G, \lambda) = \det(\lambda I_n - A(G))$, where  \( I_n \) is the \( n \times n \) identity matrix.

\begin{lemma}[\cite{Cvetkovic1995}]\label{P(G)}
Let \( G \) be a connected graph and \( e = uv \) be a bridge (cut edge) of \( G \). Then
$P(G, \lambda) = P(G - e, \lambda) - P(G - u - v, \lambda).$
\end{lemma}

\begin{lemma}\label{G(m,n)}
Suppose that $G \bullet p_l(p(u),p(v))$  is a connected graph, and \( u, v \in V(G) \) satisfy \( G - u \cong G - v \), where $p(u)-p(v)\ge1$.
 Then we have 
\(\rho(G \bullet p_l(p(u)+1,p(v)-1)) > \rho(G\bullet p_l(p(u),p(v)))\).
\end{lemma}

\renewcommand\proofname{\bf Proof}

\begin{proof}
Suppose that there are $p(u)+1$ internally disjoint copies of the paths $P_{l+1}$ at vertex $u$ in $G\bullet p_l(p(u)+1,p(v)-1)$, denoted $P_{l+1}^{j} = u u_{j1} u_{j2} \cdots u_{jl},$ $(j = 1, 2, \dots, p(u)+1)$, and $p(v)$ internally disjoint copies of the paths $P_{l+1}$ at vertex $v$  in $G\bullet p_l(p(u),p(v))$, denoted
$P^i_{l+1} = v v_{i1} v_{i2} \cdots v_{il},$ $(i = 1, 2, \dots, p(v)).$
Let \( H_1 = G\bullet p_l(p(u),p(v))\setminus {\{v, v_{i1}\}}\), and \( H_2 = G\bullet p_l(p(u)+1,p(v)-1) \setminus{\{ u, u_{j1}\}} \). Let \(H_1'\) denote the nontrivial component of \(H_1\) that is not a path, and \(H_2'\) denote the corresponding nontrivial component of \(H_2\).

Since \( G - u \cong G - v \), \( H_2' \) is a proper subgraph of \( H_1' \), and \( H_1' \) is a proper subgraph of \( G\bullet p_l(p(u)+1,p(v)-1) \). By Lemma \ref{G} (i), we get
\(
\rho(G\bullet p_l(p(u)+1,p(v)-1)) > \rho(H_1') > \rho(H_2').
\) For all \( \lambda \ge \rho(H_1') \), we have \( P(H_2', \lambda) > P(H_1', \lambda) \) by Lemma \ref{G}
(ii) and $\rho(H_1') \neq \rho(H_2')$. Let $e_1=vv_{i1}$, $e_2=uu_{j1}$. From Lemma \ref{P(G)}, we obtain that
\begin{align*}
&P(G\bullet p_l(p(u),p(v)),\lambda)\\
&=P(P_{l},\lambda)\,P(G\bullet p_l(p(u),p(v)-1),\lambda) - P(P_{l-1},\lambda)( P(P_{l},\lambda)^{p(v)-1})\ P(H_1',\lambda);\\[1ex]
&P(G \bullet p_l(p(u)+1,p(v)-1),\lambda)\\&= P(P_{l},\lambda)\,P(G\bullet p_l(p(u),p(v)-1),\lambda) )- P(P_{l-1},\lambda) (P(P_{l},\lambda)^{p(u)} )P(H_2',\lambda).
\end{align*}
Combining this with \( P(H_2', \lambda) > P(H_1', \lambda) \), we conclude
\begin{align*}
&P(G\bullet p_l(p(u),p(v)), \lambda) - P(G\bullet p_l(p(u)+1,p(v)-1), \lambda)\\
&= P(P_{l-1},\lambda) (P(P_{l},\lambda)^{p(u)-p(v)+1})(P(H_2', \lambda) - P(H_1', \lambda))>0.
\end{align*}
By Lemma \ref{G} (ii), for all \( \lambda \geq \lambda_1(G\bullet p_l(p(u)+1,p(v)-1)) \), we have \( \rho(G\bullet p_l(p(u)+1,p(v)-1))> \rho(G(u,v) \bullet p_l(p(u),p(v))) \).
This completes the proof.
\end{proof}

\begin{lemma}\label{G(k_1)}
Suppose that \(G\) is a connected graph and \(w_1,w_2,\ldots,w_k\in V(G)\) satisfy
$G-w_i\cong G-w_j$ with $ i\ne j$. Let \(G^*\) be a graph with minimum spectral radius among all graphs in $G\bullet p_l\bigl(p(w_1),p(w_2),\ldots,p(w_k)\bigr)$ such that $p(w_1)+p(w_2)+\cdots+p(w_k)\equiv r\pmod{k}$ where $ 0\le r\le k-1$. Then 
$$
G^*\in
G\bullet p_l\left(
\left(\frac{\sum_{i=1}^{k}p(w_i)-r}{k}\right)^{[k-r]},
\left(\frac{\sum_{i=1}^{k}p(w_i)-r}{k}+1\right)^{[r]}
\right).
$$

\end{lemma}

\begin{proof}
Let $S=\sum_{i=1}^{k}p(w_i)$.
We first show that, in \(G^*\), the numbers
$p(w_1),p(w_2),\ldots,p(w_k)$ differ by at most one. Suppose otherwise that there exist \(i,j\) such that
$p(w_i)\ge p(w_j)+2.$ Since
$G-w_i\cong G-w_j,$
Lemma \ref{G(m,n)} can be applied to the pair \(w_i,w_j\). Then
\[
\begin{aligned}
&\rho\left(
G\bullet p_l\left(
p(w_1),\dots,p(w_i),\dots,p(w_j),\dots,p(w_k)
\right)
\right) \\
& >
\rho\left(
G\bullet p_l\left(
p(w_1),\dots,p(w_i)-1,\dots,p(w_j)+1,
\right.\right.\left.\left.\dots,p(w_k)\right)
\right).
\end{aligned}
\]
This contradicts the minimality of \(G^*\).
Therefore,
$|p(w_i)-p(w_j)|\le1$
for all \(i,j\in\{1,\ldots,k\}\).
Now write
$S=kq+r$, $0\le r\le k-1,$ where $q=\frac{S-r}{k}.$ Since every \(p(w_i)\) is an integer and any two of them differ by at most one, each \(p(w_i)\) is either \(q\) or \(q+1\).
Let \(t\) denote the number of indices \(i\) for which
$p(w_i)=q+1.$
Then
$S=(k-t)q+t(q+1)=kq+t.$
Since \(S=kq+r\), it follows that
$t=r.$
Hence exactly \(k-r\) of the vertices \(w_1,\ldots,w_k\) have \(q\) pendent paths, while the remaining \(r\) vertices have \(q+1\) pendent paths. Consequently,
$$
G^*\in
G\bullet p_l\left(q^{[k-r]},(q+1)^{[r]}\right),
$$
which is equivalent to
$$
G^*\in
G\bullet p_l\left(
\left(\frac{\sum_{i=1}^{k}p(w_i)-r}{k}\right)^{[k-r]},
\left(\frac{\sum_{i=1}^{k}p(w_i)-r}{k}+1\right)^{[r]}
\right).
$$
The proof is complete.
\end{proof}

As an immediate consequence of Lemma \ref{G(k_1)}, we can obtain the following result.

\begin{corollary}\label{U(k_1)}
Let \(G\) be a unicyclic graph, and let \(  w_1, w_2,\dots, w_k \in V(G) \) such that \( G - w_i \cong G - w_{k-i+1} \) for all \( 1\le i \leq k \). Suppose \( G^* \) is the minimizer graph in \( G \bullet p_l(p(w_1),p(w_2),\dots,p(w_k))\) with \( p(w_1)+p(w_2)+\dots+p(w_k) \equiv r \pmod{k} \), where \( 0 \leq r \leq k-1 \). Then
$$
G^*\in
G\bullet p_l\left(
\left(\frac{\sum_{i=1}^{k}p(w_i)-r}{k}\right)^{[k-r]},
\left(\frac{\sum_{i=1}^{k}p(w_i)-r}{k}+1\right)^{[r]}
\right).
$$
\end{corollary}

\section{The minimizer graphs in \(\mathcal{U}_{n,\varphi} \)}

In this section, we characterize the structure of the graph with minimum spectral radius in $\mathcal U_{n,\varphi}$,  where $\left\lfloor \frac{2n}{3} \right\rfloor \leq \varphi \leq n-1$. 
Suppose that \( G^* \) is the minimizer graph in \( \mathcal{U}_{n,\varphi} \) and \( S(G^*) \) is an optimal dissociation set \( |S(G^*)| = \varphi\). Let \( C_q \) be the unique cycle of \( G^* \) and \( C_q = z_1z_2\cdots z_{q}z_1 \). Then we have the following claims.

\renewcommand\proofname{\bf Proof of Theorem \ref{T2}}

\renewcommand\proofname{\bf Proof}

\begin{claim}\label{1}
$G^* \in C_{q}\bullet p_{\leq2}(p(z_1),p(z_2),\dots,p(z_k))$, where $k\le q$.
\end{claim}

\begin{proof}
Otherwise, let $T_i$ be the non-trivial connected component containing $z_i$ obtained by deleting the edges in $E(C_q)$. Then we can find a  vertex $v$ on $T_i$ that is farthest from $z_i$ with $d_{G^*}(v)\ge3$. Then there are at least two paths of length greater than 1 at $v$, say $P_{s+1}, P_{t+1}$. Let $P_{s+1} = vv_{11}\cdots v_{1s}$ and $P_{t+1} = vv_{21}\cdots v_{2t}$, where $s\ge t\ge 1$. If $t\le 2$, then let $G_1=G^* - vv_{21} + v_{1s} v_{21}$. 
If $s\geq t\ge 3$, let 
\[
G_1 =
\begin{cases}
G^* - vv_{21} + v_{1s} v_{21}, & \text{if } t \equiv 0\pmod{3}; \\[4pt]
G^* - v_{21} v_{22} + v_{1s}v_{22}, & \text{if } t \equiv 1 \pmod{3}; \\[4pt]
G^* - v_{22} v_{23} + v_{1s} v_{23}, & \text{if } t \equiv 2 \pmod{3}.
\end{cases}
\]
Then we have \(n(G_1) = n(G^*)\) and \(\varphi(G_1) = \varphi(G^*)\) by Lemma \ref{LS(G)}, and \( G_1 \in \mathcal{U}_{n,\varphi} \). By Lemma \ref{LG(s,t)}, we have \(\rho(G^*) > \rho(G_1)\), a contradiction. By repeating this process, we can obtain that $G^* \in C_{q}\bullet p_{l_j}(p(z_1),p(z_2),\dots,p(z_k))$ for \(l_j \geq 0 \).

In the following, we prove that $l_j\le2$. Suppose, to the contrary, that there exists at least a path of length greater than 3 at $z_i$, where $i\in[1,k]$. If there exist at least two paths of length greater than 3 at $z_i$ for $i\in[1,k]$, by performing an operation similar to that on $G_1$, we obtain a graph with smaller spectral radius than $G^*$, which contradicts the minimality of $G^*$.
If there is a path of length greater than 3 at $z_i$, for convenience, assume that \( P_{s+1}= z_i u_1 u_2 \dots u_s \) with \( s \geq 3 \). Let \( G_2 \) be the graph obtained from \( G^* \) by subdividing the edge \( z_i z_{i+1} \in E(C_q) \) with $3s'$ times and deleting the vertices \( \{ u_s, u_{s-1}, \dots, u_{s-3s'+1} \} \), where \( 3s' \leq s \). Then we obtain \( n(G_2) = n(G^*) \), \( \varphi(G_2) = \varphi(G^*) \), and \( G_2 \in \mathcal{U}_{n,\varphi} \). However, by Lemmas \ref{G}(i) and \ref{G_w}, we have \( \rho(G^*) > \rho(G_2) \), a contradiction. 
\end{proof}

\begin{claim}\label{W}
 If \( z_i \notin S(G^*) \), then \(z_{i+1} \in S(G^*) \) and \(
z_{i+2} \in S(G^*)\), where $1\le i\le q-2$.
\end{claim}

\begin{proof}
On the contrary, $z_{i+1} \notin S(G^*)$ or $z_{i+2} \notin S(G^*)$. If $z_{i+1} \notin S(G^*)$. Let $G_3$ be the graph obtained from $G^*$ by subdividing the edge $z_iz_{i+1}$ twice (the new vertices are $z',z''$) and deleting two pendant vertices. since $z_i,z_{i+1}\notin S(G^*)$, $z',z''\in S(G^*)$. By Lemma \ref{LS(G)} and Claim \ref{1}, we have \(\varphi(G_3)= \varphi(G^*) \) and \(n(G_3)= n(G^*) \), and so \( G_3 \in \mathcal{U}_{n,\varphi} \). By  Lemmas \ref{G}(i) and \ref{G_w}, we obtain that \(\rho(G^*) > \rho(G_3),
\) a contradiction. If $z_{i+2} \notin S(G^*)$, then $z_{i+1} \in S(G^*)$. Let $G_4$ be the graph obtained from $G^*$ by subdividing the edge $z_iz_{i+1}$ once (the new vertex is $z$) and deleting one pendant vertex. Since \( z_{i}, z_{i+2} \notin S(G^*) \) and \( z ,z_{i+1} \in S(G^*) \), by Lemma \ref{LS(G)} and Claim \ref{1}, we have \(\varphi(G_4)= \varphi(G^*) \) and \(n(G_4)= n(G^*) \), and so \( G_4 \in \mathcal{U}_{n,\varphi} \). By  Lemmas \ref{G}(i) and \ref{G_w}, we obtain \(\rho(G^*) > \rho(G_4),
\) a contradiction. Hence \( z_{i+1}, z_{i+2} \in S(G^*) \). 
\end{proof}

By Claim \ref{W}, we obtain $q=3k$. Let \( V(C_{3k})= X \cup Y\), where \( Y =S(G^*)\cap V(C_{3k})\). Without loss of generality, assume $z_1 \notin S(G^*)$. Then we have \( X =\{z_1, z_4, \dots, z_{3k-2}\}\). For convenience, write $X=\{w_1,\dots,w_k\}$. If $z_i\in Y$, then \( d(z_i) = 2 \), which means $N_{G^*}(z_i)\setminus V(C_{3k})=\emptyset$. Note that $S^*(G^*)=V(G^*)\setminus V(C_{3k})$
by Claim \ref{1}. Then \( X =V(G^*)\setminus S(G^*)\). Since \( \varphi(G^*) = \varphi \), we have \( \varphi =n- k \). In what follws, we consider two cases according to the parity of $ \varphi$.

\noindent\textbf{Case 1.}
\( \varphi\) is even.

Observe that \( n-k\) is even. Thus, $n-3k$ is even. We can obtain the following fact.

\begin{fact}\label{fact}
$G^*\in C_{3k}\bullet p_{2}(p(w_1),p(w_2),\dots,p(w_k))$.
\end{fact}

\renewcommand\proofname{\bf Proof }
\begin{proof} 
Suppose to the contrary that $G^*\notin C_{3k} \bullet p_{2}(p(w_1),p(w_2),\dots,p(w_k))$. Let $G_{5}\in C_{3k}\bullet p_{2}(p(w_1),p(w_2),\dots,p(w_k))$. Without loss of generality, let $P_2=u_1u_2$ attached to $w_i ~(w_i\in X)$ of $G_5$. Since \( G_5 \) contains the cycle \( C_{3k} \), we have \( \rho(G_5) > 2 \).
Suppose that $\boldsymbol{x}=(x_1,x_2,\dots,x_n)^T$ is  the Perron eigenvector of $G_{5}$. By the characteristic equation $\rho(G_{5}) x_{v_i}=\sum\limits_{v_j \in N_{G_5}(v_i)} x_{v_j}$, we have 

\[
\begin{cases}
\rho(G_{5}) x_{u_1}=x_{w_i}+x_{u_2},\\
\rho(G_{5}) x_{u_2}=x_{u_1},
\end{cases}
\]
from which we obtain that $x_{w_i} = \left( \rho(G_{5}) - \frac{1}{\rho(G_{5})} \right) x_{u_1} > x_{u_1}$, where $\rho(G_{5})>2$.
Let $x_{w^*} = \max\{ x_{w_i}|$ $ w_i \in X \}$. Then we have $x_{w^*}\ge x_{w_i} > x_{u_1}$. Let $G_{6}=G_{5}-u_1u_2+w^*u_2$. Thus 
there exists at least a pair of pendant vertices $u_1,u_2$ attached at $w_i,w^*$ ($w_i$ and $w^*$ can be the same vertex), respectively. Hence, $G_{6}\notin C_{3k}\bullet p_{2}(p(w_1),p(w_2),\dots,p(w_k))$. However, by Lemma \ref{N(u)}, we have $\rho(G_{6})>\rho(G_{5})$, which leads to a contradiction.
\end{proof}
Then by Fact \ref{fact}, we have \( C_{3k} - w_i \cong C_{3k}- w_{k-i+1} \) for \( 1\le i \leq k \). Note that \(|S^*(G^*)| = n - 3k\). 
Let
\( \frac{n-3k}{2} \equiv r \pmod{k} \). Then by Corollary \ref{U(k_1)}, we obtain that 
\begin{align}\tag{$*$}\label{*}
G^* \in C_{3k} \bullet p_2\left( \left( \frac{n - 2r - k}{2k} - 1 \right)^{[k-r]},\; \left( \frac{n - 2r - k}{2k} \right)^{[r]} \right).
\end{align}

\noindent\textbf{Case 2.}
\( \varphi\) is odd. 

Recall that $X=\{w_1,\ldots,w_k\}$. Thus the sum of all pendant vertices attached to every $w_i$ is odd. Without loss of generality, suppose $w_1$ is incident to a pendant vertex $u$. Denote $G(v)=G+uv$, where $v$ is a pendant vertex attached to $u$. Then $\varphi(G(v))=\varphi(G)+1$ by Lemma \ref{LS(G)}.  This implies $G(v)\in\mathcal{U}_{n+1,\varphi+1}$, and $\varphi(G(v))$ is even. Suppose that $G^1(v)\in C_{3k}\bullet p_2(p(w_1)+1, p(w_2), \dots, p(w_k))$. Thus by Fact \ref{fact}, we have
\begin{equation}\label{eq:1}\rho(G(v))\ge \rho(G^1(v)).\end{equation}
Let $\frac{n-3k-1}{2} \equiv r \pmod{k}$ and ${G}^2(v) \in C_{3k} \bullet p_2\left(\left(\frac{n-2r-k-1}{2k}-1\right)^{[k-r-1]},\left(\frac{n-2r-k-1}{2k}\right)^{[r+1]}\right)\in\mathcal{U}_{n+1,\varphi+1}$. Then by Corollary \ref{U(k_1)}, we obtain that 
\begin{equation}\label{eq:2}
\rho(G^1(v))\ge \rho(G^2(v)).
\end{equation}

Set $a = \frac{n - 2r - k - 1}{2k}$ and let $G_{7}$ be obtained from $G^2(v)$ by deleting a pendant vertex that lies on an arbitrary pendant path starting at $w_i$, which is a vertex with $a$  pendant paths in $G^2(v)$. Then  
$G_{7}\in C_{3k}\bullet p_2\left(\left(\frac{n-2r-k-1}{2k}-1\right)^{[k-r]}_1, \left(\frac{n-2r-k-1}{2k}\right)^{[r]}\right)(\subseteq \mathcal{U}_{n,\varphi}).$
Without loss of generality, suppose $w_k$ is incident to a pendant vertex $u'$ in $G_{7}$. Denote $G_{7}(v')=G_{7}+u'v'$, where $v'$ is a new pendant vertex attached to $u'$. Then $\varphi(G_{7}(v'))=\varphi(G_{7})+1$ by Lemma \ref{LS(G)} and $\varphi(G_{7}(v'))$ is even. Note that $G_{7}(v')\in C_{3k} \bullet p_2\left(\left(\frac{n-2r-k-1}{2k}-1\right)^{[k-r-1]},\left(\frac{n-2r-k-1}{2k}\right)^{[r+1]}\right)(\subseteq\mathcal{U}_{n+1,\varphi+1})$. Thus, by inequations (\ref{eq:1}) and (\ref{eq:2}), we have $\rho(G(v))\ge\rho(G_{7}(v'))$.
From Lemma \ref{G} (ii), we can obtain $P(G(v),\lambda)<P(G_{7}(v'),\lambda)$. Since $\varphi(G-u)=\varphi(G_{7}-u')$ is even, we get $\rho(G-u)>\rho(G_{7}-u')$ by Fact \ref{fact} and Corollary \ref{U(k_1)}. Therefore, we have $P(G-u,\lambda)<P(G_{7}-u',\lambda)$ by Lemma \ref{G} (ii).
Let $e_1=uv\in E(G(v))$ and $e_2=u'v'\in E(G_{7}(v'))$. Then from Lemma \ref{P(G)}, we obtain
\begin{align*}
P(G(v),\lambda)&=P(G(v)-e_1,\lambda)-P(G(v)-u-v,\lambda),\\
&= \lambda P(G,\lambda)-P(G-u,\lambda),
\end{align*}
\begin{align*}
P(G_{7}(v'),\lambda)&=P(G_{7}(v')-e_2,\lambda)-P(G_{7}(v')-u'-v',\lambda),\\
&= \lambda P(G_{7},\lambda)-P(G_{7}-u',\lambda).
\end{align*}
It follows that  $$\lambda(P(G_{7}, \lambda) - P(G, \lambda)) 
= \big(P(G_{7}(v'), \lambda) - P(G(v), \lambda)\big) + \big( P(G_{7} - u', \lambda) - P(G - u, \lambda) \big)>0.$$
From Lemma \ref{G} (ii), for all \( \lambda \geq \rho(G) \), we obtain that $\rho(G)>\rho(G_{7})$.
Recall that $G_{7}\in C_{3k}\bullet p_2\left(\left(\frac{n-2r-k-1}{2k}-1\right)^{[k-r]}_1,\left(\frac{n-2r-k-1}{2k}\right)^{[r]}\right)$.
Therefore \begin{align}\tag{$**$}\label{**}
G^*\in C_{3k}\bullet p_2\left(\left(\frac{n-2r-k-1}{2k}-1\right)^{[k-r]}_1,\left(\frac{n-2r-k-1}{2k}\right)^{[r]}\right).
\end{align}

We now proceed to prove the main results of Theorems \ref{T2}-\ref{P2}.

\renewcommand\proofname{\bf Proof of Theorem \ref{T2}}

\begin{proof}

(i). When $\varphi = n-1$, we have $|X|=k = n-\varphi = 1$ and $X=\{w_1\}$. If $\varphi$ is even, then
by (\ref{*}) we have $G^*=C_3\bullet p_2\left(\frac{n-3}{2}\right)$. If $\varphi$ is odd, then
by (\ref{**}), we obtain $G^*= C_3\bullet p_2\left(\left\lfloor\frac{n-3}{2}\right\rfloor_1\right)$.

(ii). When $\varphi = n-2$, we have $k=2$ and $X=\{w_1, w_2\}$. If $\varphi$ is even, then by (\ref{*}), we get $G^* = C_6 \bullet p_2\left(\left\lfloor \frac{n-6}{4} \right\rfloor, \left\lceil \frac{n-6}{4} \right\rceil\right)$. If $\varphi$ is odd, then by (\ref{**}), we have $G^* = C_6\bullet p_2\left(\left\lfloor \frac{n-6}{4} \right\rfloor_1, \left\lceil \frac{n-6}{4} \right\rceil\right)$.
\end{proof}

\renewcommand\proofname{\bf Proof of Theorem \ref{T3}}
\begin{proof}
Let $\varphi = n - k \le n- 3$. Then we have $|X|=k\ge3$, and so \(C_{3k}=S_2(C_k) \) and $X=V(C_k)$ by Claim \ref{W}. Therefore, by (\ref{*}) and (\ref{**}), we have
\[
G^* \in
\begin{cases}
\mathcal{G}^*_1=S_2(C_k) \bullet p_2\left( \left( \frac{n-2r-k}{2k} - 1 \right)^{{[k-r]}} ,\left( \frac{n-2r-k}{2k} \right)^{{[r]}}\right),  &{\varphi} \text{ is even};\\[5pt]
\mathcal{G}^*_2=S_2(C_k) \bullet p_2\left( \left( \frac{n-2r-k-1}{2k} - 1 \right)^{{[k-r]}}_{1}  ,\left( \frac{n-2r-k-1}{2k} \right)^{{[r]}}\right), &{\varphi} \text{ is odd.}
\end{cases}
\]
\end{proof}

\renewcommand\proofname{\bf Proof of Theorem \ref{P2}}
\begin{proof}
(i). Let \( n-4 \) be even. Then for \( r \in \{0,1,3\} \), the conclusion is obvious.  Let $X=V(C_4) = \{w_1,\dots,w_4\}$, \(
G_1' = S_2(C_4) \bullet p_2\left( \frac{n-16}{8}, \frac{n-16}{8}, \frac{n-8}{8}, \frac{n-8}{8} \right)
\) and 
\(
G_1'' = S_2(C_4) \bullet p_2\left(\frac{n-16}{8}, \frac{n-8}{8}, \frac{n-16}{8}, \frac{n-8}{8} \right).
\) If \( r=2 \), we have \( G^* \in\{ G_1',G_1''\} \) by Theorem \ref{T3}.  Let $V(S_2(C_4)\setminus V(C_4)) = \{z_1,\dots,z_8\}$ and 
$N_i= N_{G_1'}(w_i)\setminus N_{S_2(C_4)}(w_i),$ where $1\le i\le 4$. Then $Q_2(G_1') =\{u_1,\dots ,u_4\}= \bigcup_{i=1}^4 N_i$ and $
P(G_1')  =\{v_1,\dots ,v_4\}= \bigcup_{u_i\in N_i} N_{G_1'}(u_i)\setminus \{w_i\}$ for $1\le i\le 4$. Then \( V(G_1') \) can be partitioned as
\(
V(C_4) \cup V(S_2(C_4) \setminus V(C_4)) \cup Q_2(G_1') \cup P(G_1').
\) Therefore, \( A(G_1') \) has an equitable quotient matrix $B_\Pi(G_1')$.
Similarly, \( A(G_1'') \) has an equitable quotient matrix $B_\Pi(G_1'')$. The specific quotient matrix $B_{\Pi}(G_1')$ and its characteristic polynomial are presented in Appendix (i). By direct calculation,  we have
\[
P(B_\Pi(G_1''), \lambda) - P(B_\Pi(G_1'), \lambda) = \lambda^2 (\lambda-1)^4 (\lambda+1)^4 > 0.
\]
By Lemmas \ref{G} (ii) and \ref{B}, we obtain that  \( \rho(G_1'') < \rho(G_1') \).
Thus
\[
G^* = S_2(C_4) \bullet \mathfrak{b}p_2\left( \left( \frac{n-16}{8} \right)^2, \left( \frac{n-8}{8} \right)^2 \right).
\]

(ii). Let \( n-4 \) be odd. For \( r \in \{0,3\} \), the conclusion is obvious. If  $r=1$, let \( a=\frac{n-7}{8} \), and \(
G_2' = S_2(C_4) \bullet p_2\left(\left( a-1\right)_1,  a-1, a-1,  a \right),
\) \(
G_2'' = S_2(C_4) \bullet p_2\left( a-1,\left( a-1\right)_1, a-1, a\right). \)   Then \( G ^*\in\{ G_2',G_2''\} \) by Theorem \ref{T3}. When \( a=1\), then \( \rho(G_2') = 2.13746 > \rho(G_2'') = 2.12157 \).
When \( a=2\), we have \( \rho(G_2') = 2.34760 > \rho(G_2'') = 2.33159 \). In the following, let \( a\geq 3 \). Since \( G_2' \) contains a star graph \( S_{a+3} \) with \( \rho(S_{a+3}) = \sqrt{a+2} \), we have \(\rho( G_2')  \geq \sqrt{5} \).
 Using the same method as in  (i), we get \( A(G_2') \) has an equitable quotient matrix $B_\Pi(G_2')$ and \( A(G_2'') \) has an equitable quotient matrix $B_\Pi(G_2'')$. The specific quotient matrix $B_{\Pi}(G_2')$ and its characteristic polynomial are given in Appendix (ii). For $\lambda\ge \sqrt{5}$, we get 
\begin{align*}
P(B_\Pi(G_2''), \lambda) - P(B_\Pi(G_2'), \lambda)
= \lambda^{11} - 5\lambda^9 + 10\lambda^7 - 10\lambda^5 + 5\lambda^3 - \lambda \geq 0.
\end{align*}

By Lemmas \ref{G} (ii) and \ref{B}, we obtain that \( \rho(G_2'') < \rho(G_2') \).
Then
\[
G^* = S_2(C_4) \bullet \mathfrak{b}p_2\left( \left( \frac{n-15}{8} \right)^{[3]}_1, \left( \frac{n-7}{8} \right)^{[1]} \right).
\]

We next consider the case $r=2$. Let 
\(
G_3' = S_2(C_4) \bullet p_2\left(\left( \frac{n-17}{8}\right)_1,\frac{n-17}{8}, \frac{n-9}{8}, \frac{n-9}{8} \right),
\) and let $
G_3'' = S_2(C_4)\bullet p_2\left(\left(\frac{n-17}{8}\right)_1,\,\frac{n-9}{8},\frac{n-17}{8},\frac{n-9}{8}\right)
$.
Then \( G \in\{ G_3',G_3''\} \) by Theorem \ref{T3}. With the aid of the quotient matrices, we compute $B_{\Pi}(G_3')$ and $B_{\Pi}(G_3'')$ directly. The explicit quotient matrix $B_{\Pi}(G_3')$ and its characteristic polynomial are provided in Appendix (iii).
Since \( G_3' \) contains the cycle \( C_{12} \), we have \( \rho(G_3') > 2 \).
By direct calculation, we obtain
\begin{align*}
P(B_\Pi(G_3''), \lambda) - P(B_\Pi(G_3'), \lambda) 
= \lambda^9 - 4\lambda^7 + 6\lambda^5 - 4\lambda^3 + \lambda\geq 0.
\end{align*}

By Lemmas \ref{G} (ii) and \ref{B}, we have \( \rho(G_3'') < \rho(G_3') \).
Thus
\[
G^* = S_2(C_4) \bullet \mathfrak{b}p_2\left( \left( \frac{n-17}{8} \right)^{[2]}_1, \left( \frac{n-9}{8} \right)^{[2]} \right).
\]
This completes the proof.
\end{proof}

\section{Final remarks}
Suppose that \( G^* \) is the minimizer graph in \( \mathcal{U}_{n,\varphi} \), where  $\varphi \le n-3.$
Let \( k = n - \varphi \), \( \left\lfloor \frac{n-3k}{2} \right\rfloor \equiv r \pmod{k} \) and 
let $X=V(C_k)$. Theorem \ref{T3} of this paper obtains the structure of graphs attaining the minimal spectral radius as follows:
\[
G^* \in
\begin{cases}
\mathcal{G}^*_1=S_2(C_k) \bullet p_2\left( \left( \frac{n-2r-k}{2k} - 1 \right)^{{[k-r]}} ,\left( \frac{n-2r-k}{2k} \right)^{{[r]}}\right),  &{\varphi} \text{ is even};\\[5pt]
\mathcal{G}^*_2=S_2(C_k) \bullet p_2\left( \left( \frac{n-2r-k-1}{2k} - 1 \right)^{{[k-r]}}_{1}  ,\left( \frac{n-2r-k-1}{2k} \right)^{{[r]}}\right), &{\varphi} \text{ is odd.}
\end{cases}
\]

Theorems \ref{T2}, \ref{T4} and \ref{P2} determine the graphs with
fixed order $n$ and dissociation number $\varphi\in \{n-1, n-2, n-3,n-4\}$ having minimum spectral radii. It is natural to consider the corresponding extremal graphs for the case $\varphi\le n-5$. The Table \ref{tab:minimizer_G^*} below records the relevant results when $\varphi= n-5$, that is $k=5$ and $r\in [0,4]$. By Theorem \ref{T3}, if $\varphi$ is even, we obtain that the extremal graph is uniquely determined when $r=0,1,4$. For $r=2,3$, there exist two or more graphs in $\mathcal{G}^*_1$. Similarly, if $\varphi$ is odd, there exist two or more graphs in $\mathcal{G}^*_2$ for $r=1,2,3$. In Table \ref{tab:minimizer_G^*} below, when $r$ is determined, we take small values of $n$ to find the  extremal graph $G^*$.

\setlength{\tabcolsep}{25pt}
\renewcommand{\arraystretch}{2.1}
\begin{table}[htbp]
\centering
\caption{\small The minimizer graph $G^*$ with $k=5$, and its spectral radius ($\left\lfloor \frac{n-15}{2} \right\rfloor \equiv r \pmod{5}$).}
\label{tab:minimizer_G^*}
\vspace{4pt}
\begin{tabular}{|c|c|c|c|c|}
\hline
$\varphi=n-5$ & $r$ & $n$ & $\mathcal{G}^*_1 ~\text{or} ~\mathcal{G}^*_2$ & $ $ \\
\hline
\multirow{4}{*}{even}
& \multirow{2}{*}{2}
& \multirow{2}{*}{19}
& \includegraphics[scale=0.08]{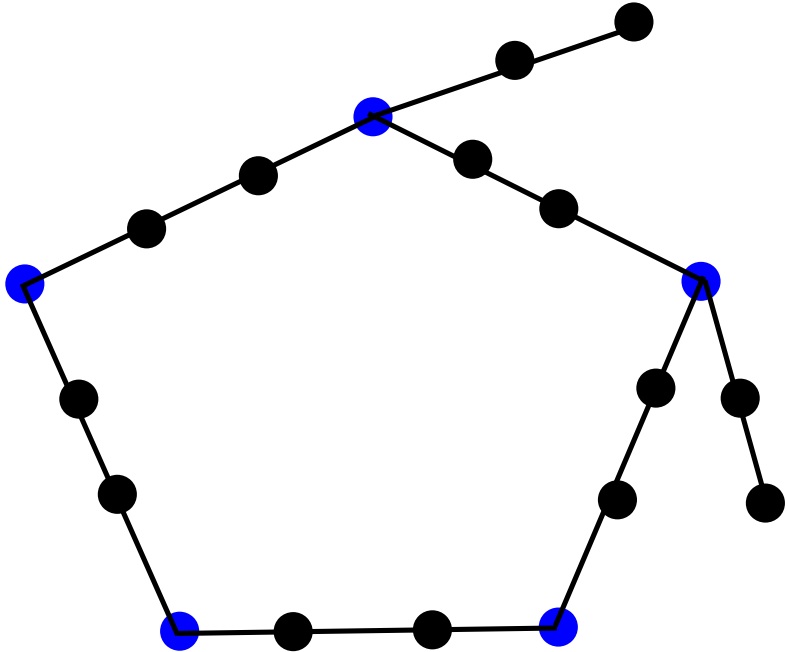} &$ \rho=2.15041$  \\
\cline{4-5}
& &
& \includegraphics[scale=0.08]{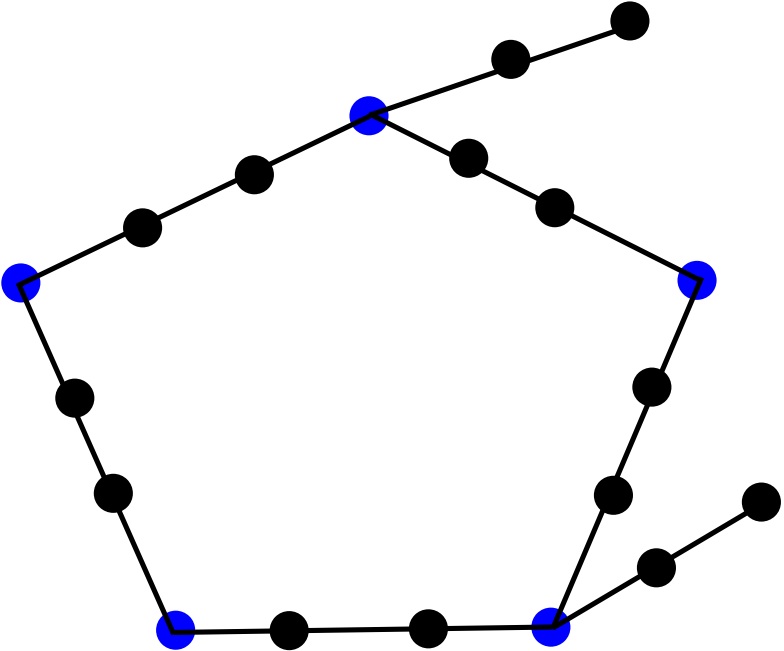}&  $\rho(G^*)=2.12322$ \\
\cline{2-5}
& \multirow{2}{*}{3}
& \multirow{2}{*}{21}
& \includegraphics[scale=0.08]{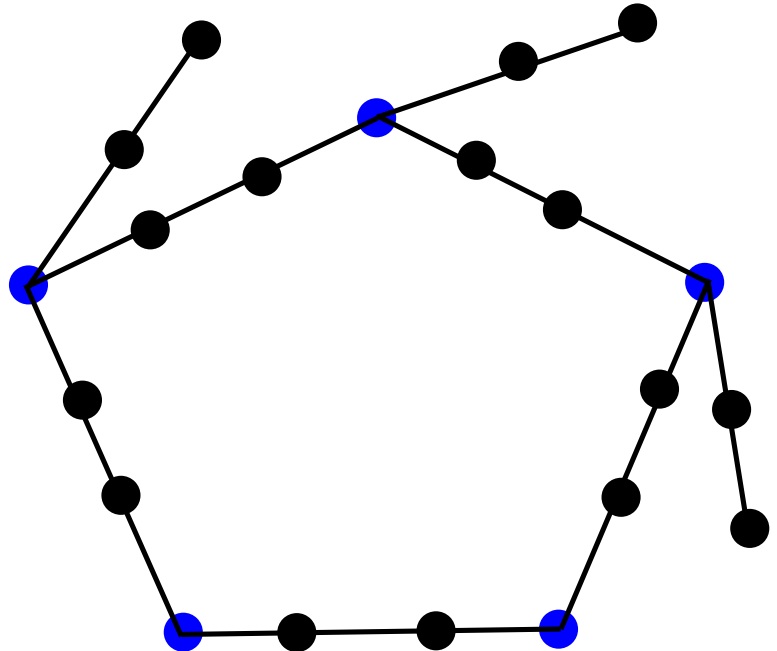} & $ \rho=2.17659$  \\
\cline{4-5}
& &
& \includegraphics[scale=0.08]{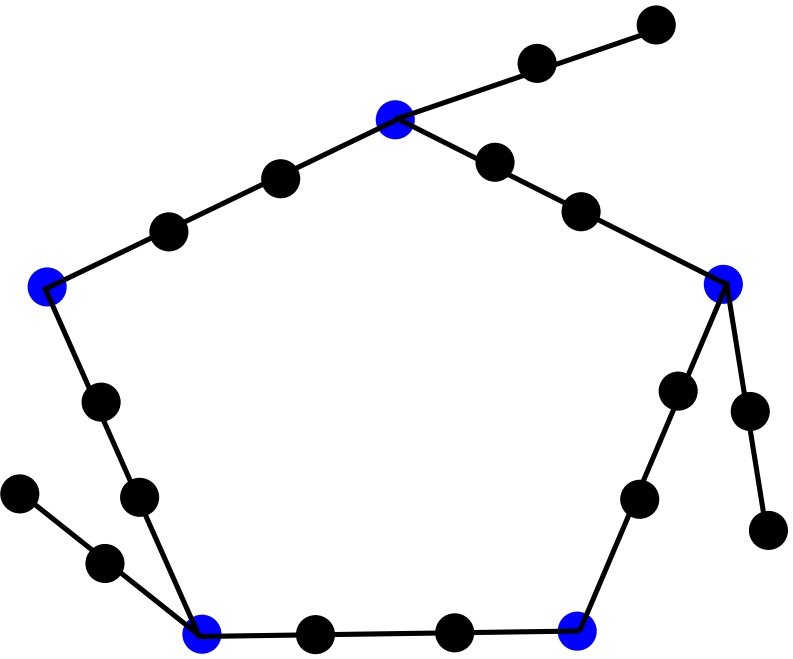}&   $\rho(G^*)=2.16109$ \\
\hline
\multirow{8}{*}{odd}
& \multirow{2}{*}{1}
& \multirow{2}{*}{18}
& \includegraphics[scale=0.08]{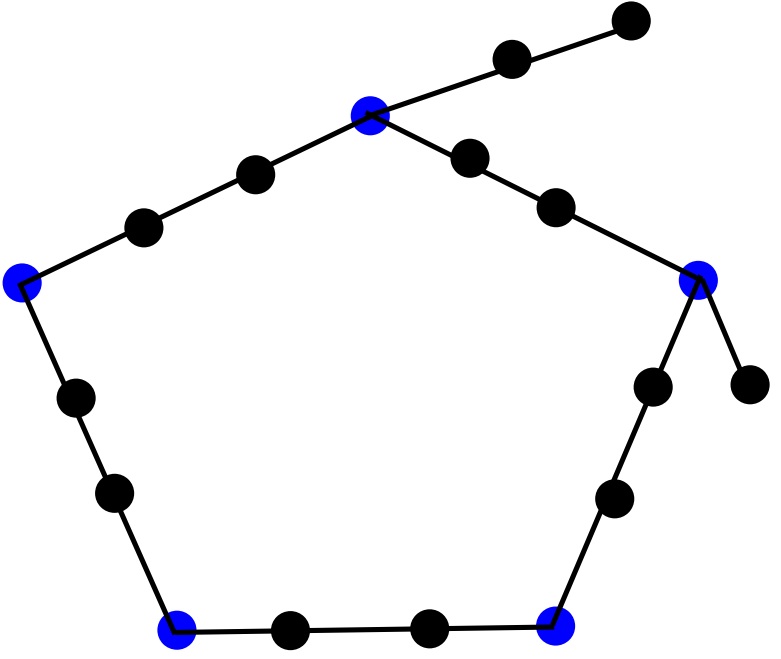} & $ \rho=2.13208$  \\
\cline{4-5}
& &
& \includegraphics[scale=0.08]{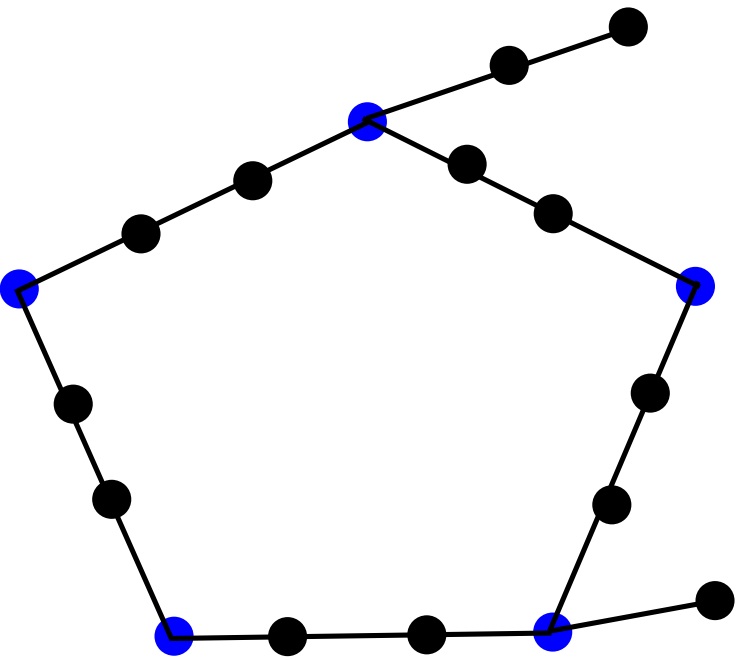}& $\rho(G^*)=2.11040 $\\
\cline{2-5}
& \multirow{4}{*}{2}
& \multirow{4}{*}{20}
& \includegraphics[scale=0.08]{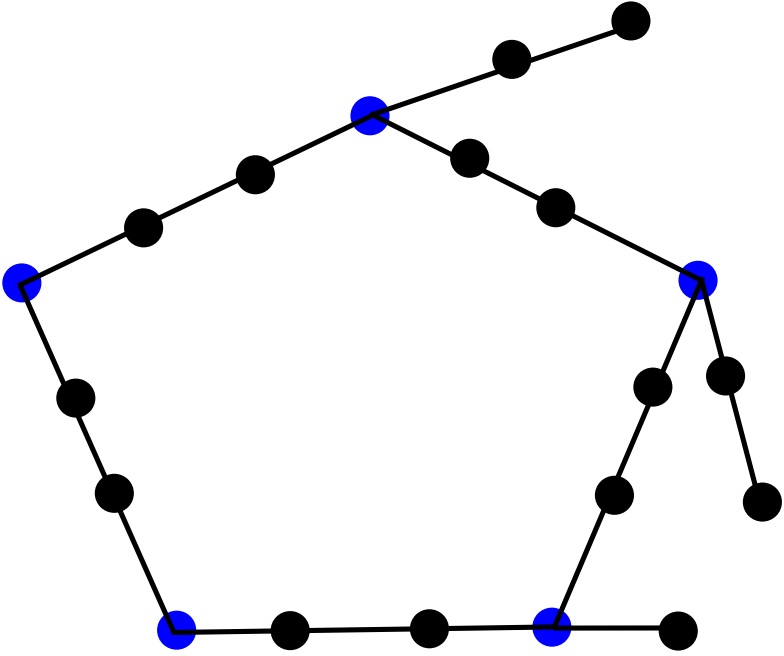} & $ \rho=2.16682$  \\
\cline{4-5}
& &
& \includegraphics[scale=0.08]{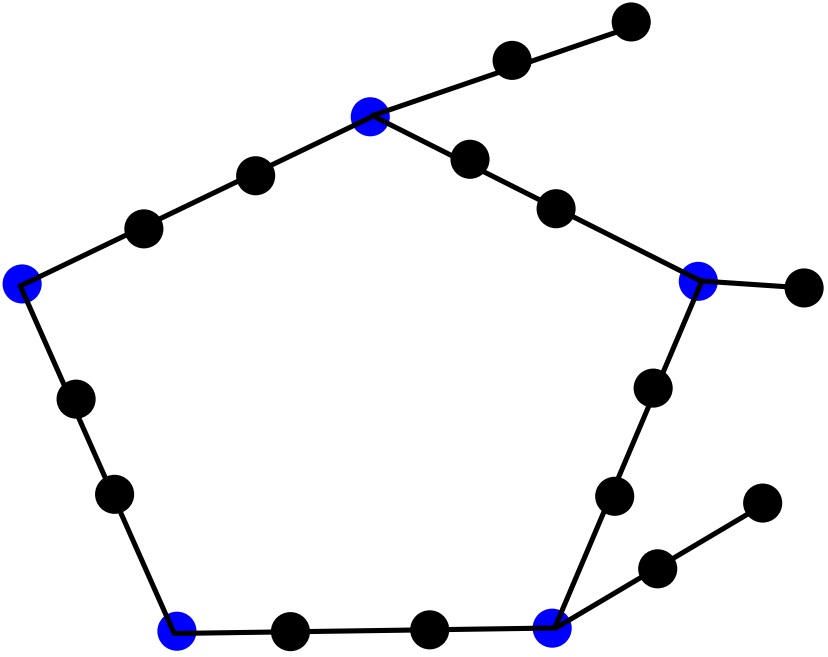}& $ \rho=2.15997$ \\
\cline{4-5}
& &
& \includegraphics[scale=0.08]{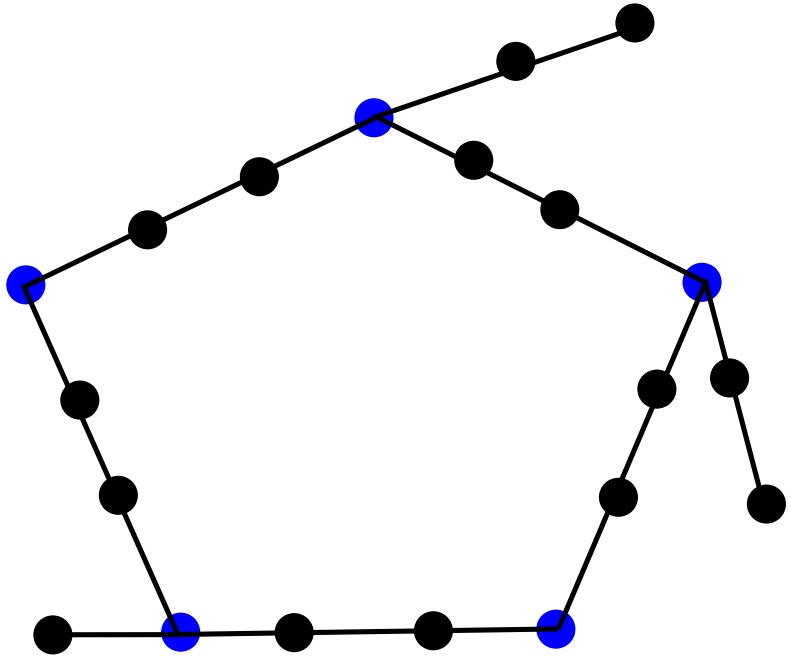} & $ \rho=2.15644$ \\
\cline{4-5}
& &
& \includegraphics[scale=0.08]{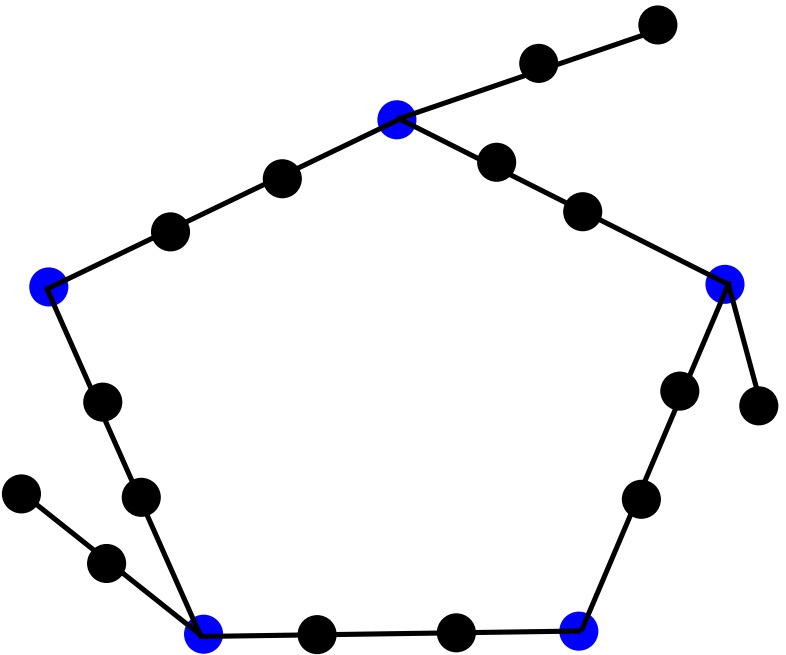} & $\rho(G^*)=2.14707$\\
\cline{2-5}
& \multirow{2}{*}{3}
& \multirow{2}{*}{22}
& \includegraphics[scale=0.08]{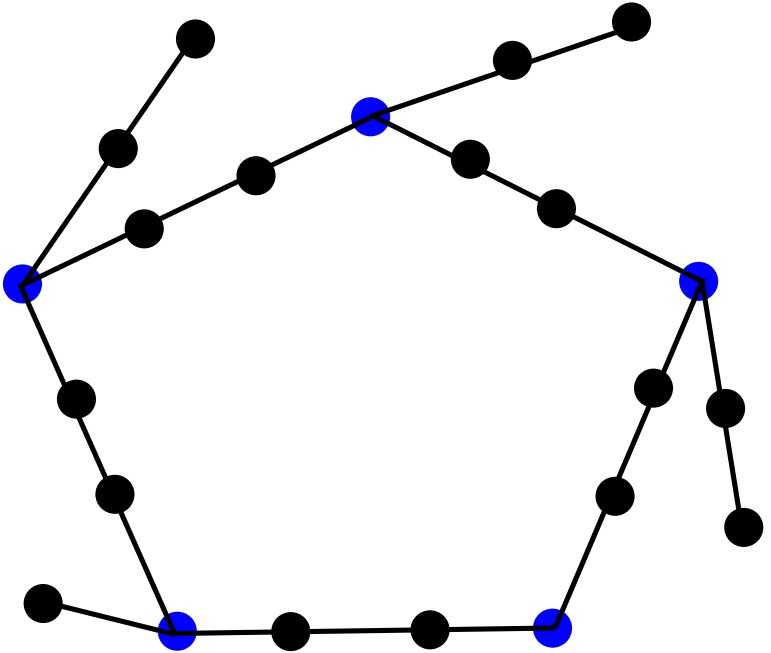} & $ \rho=2.18728$  \\
\cline{4-5}
& &
& \includegraphics[scale=0.08]{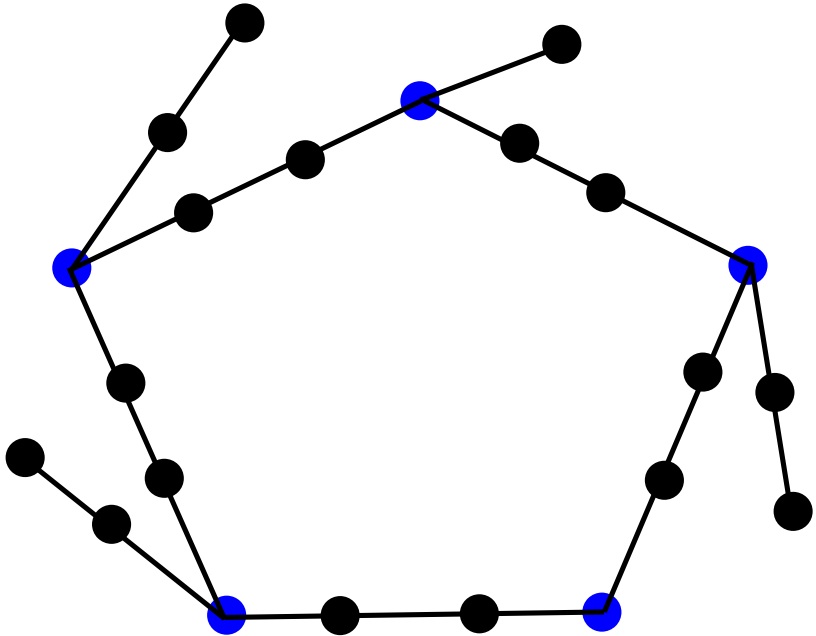}& $\rho(G^*)=2.18252 $\\
\hline
\end{tabular}
\end{table}

Based on the foregoing analysis, we propose the following conjecture to inspire future related studies.

\begin{conjecture}
Suppose that \( G^* \) is the minimizer graph in \( \mathcal{U}_{n,\varphi} \) with \( \left\lfloor \frac{2n}{3} \right\rfloor \leq \varphi \leq n-5 \).
Let $X=V(C_k)$ and let \( k = n-\varphi \) and \( \left\lfloor \frac{n-3k}{2} \right\rfloor \equiv r \pmod{k} \). Then
\[
G^* =
\begin{cases}
S_2(C_k) \bullet \mathfrak{b}p_2\left( \left( \frac{n-2r-k}{2k} - 1 \right)^{[k-r]}, \left( \frac{n-2r-k}{2k} \right)^{[r]} \right), & n-k \text{ is even}; \\
S_2(C_k) \bullet \mathfrak{b}p_2\left( \left( \frac{n-2r-k-1}{2k} - 1 \right)^{[k-r]}_1, \left( \frac{n-2r-k}{2k} \right)^{[r]} \right), & n-k \text{ is odd}.
\end{cases}
\]
\end{conjecture}

\vskip 0.382 in
\noindent

\clearpage

\section*{Declaration of Competing Interest}

The authors declare that they have no known competing financial interests or personal
relationships that could have appeared to influence the work reported in this paper.

\vskip 0.382 in
\noindent

\section*{Acknowledgements}
This research is supported by National Natural Science Foundation of China (No. 12401468), Tian shan Talent Training Program (No. 2024TSYCQNTJ0001), China Postdoctoral Science Foundation (No. 2025M783100), the Natural Science Foundation of Xinjiang Uygur Autonomous Region (No. 2023D01C165).

\vskip 0.382 in
\clearpage

\noindent
\textbf{\Large Appendix}
 
(i). Below are the matrix \( B_{\Pi}(G_1') \) and its characteristic polynomial \( P(B_{\Pi}(G_1'), \lambda) \).
\begin{figure}[htbp]
    \centering    \includegraphics[width=0.7\linewidth]{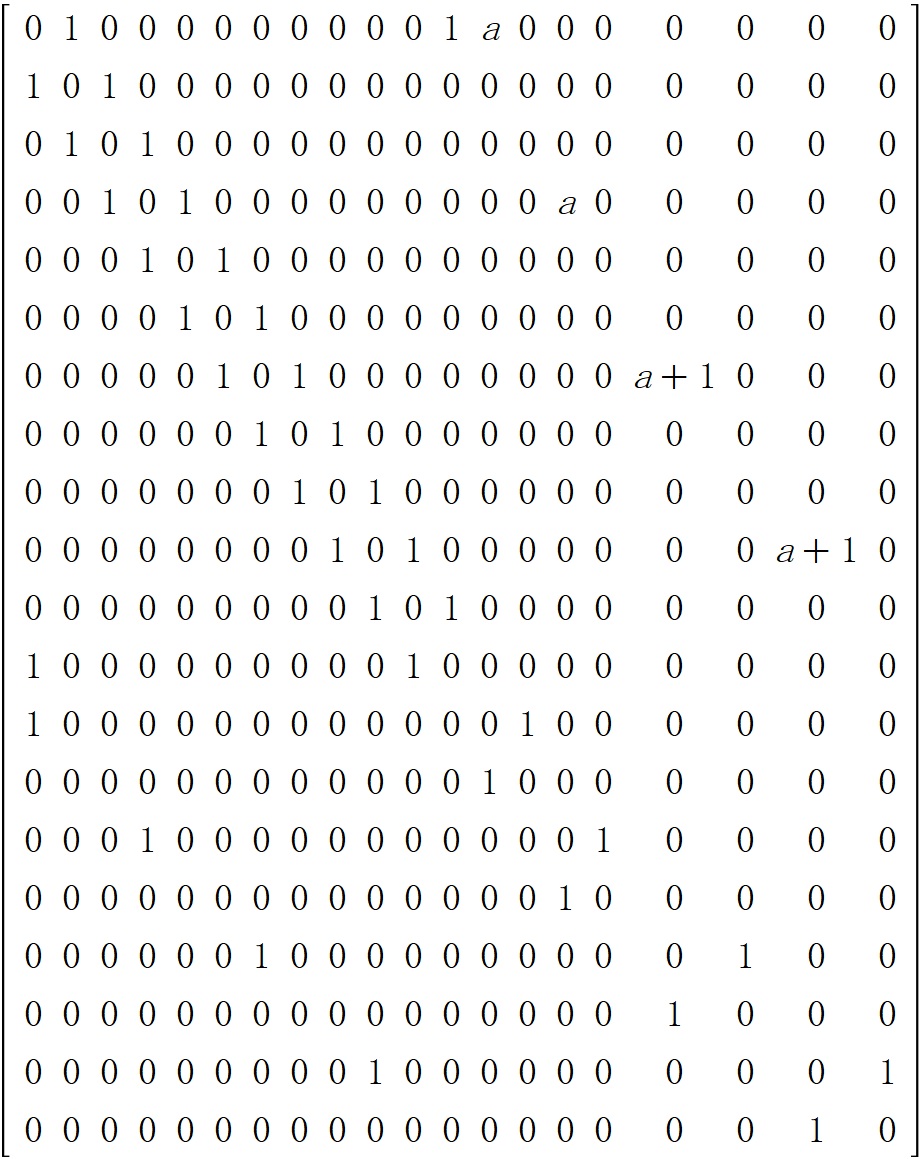}
    \begin{center}
    \small
    The vertices corresponding to the rows/columns of the matrix 
   $B_{\Pi}(G_1')$ are respectively:
    \(
    w_1,\, z_1,\, z_2,\, w_2,\, z_3,\, z_4,\, w_3,\, z_5,\, z_6,\, w_4,\, z_7,\, z_8,\, u_1,\, v_1,\, u_2,\, v_2,\, u_3,\, v_3,\, u_4,\, v_4.
    \)
  \end{center}
\end{figure}
\begin{align*}
&P(B_\Pi(G_1'), \lambda)=\!\lambda^{20}-\!(18 + \!4a)\lambda^{18} +\! (6a^2 +\! 58a + \!135)\lambda^{16} -\! (4a^3 +\! 66a^2 +\! 338a +\! 552)\lambda^{14} \\
&+ \!(a^4 +\! 30a^3 +\! 277a^2 +\! 1028a +\! 1355)\lambda^{12} - \!(4a^4 +\! 80a^3 +\!572a^2 +\! 1780a +\! 2075)\lambda^{10}+\!(6a^4 \\&+\! 100a^3 +\! 628a^2 +\! 1794a +\! 1989)\lambda^8 -\! (4a^4 +\! 60a^3 +\! 358a^2 +\! 1018a +\! 1154)\lambda^6+ \!(a^4 +\!14a^3\\& +\! 89a^2 + \!288a +\! 368)\lambda^4 -\! (4a^2 +\! 28a +\! 49)\lambda^2.
\end{align*}
Similarly, the characteristic polynomial of $B_\Pi(G_1'')$ is as follows.
\begin{align*}
&P(B_\Pi(G_1''), \lambda) =\lambda^{20} -\! (18 +\! 4a)\lambda^{18} +\! \bigl(6a^2 +\! 58a +\! 135\bigr)\lambda^{16} -\! \bigl(4a^3 +\! 66a^2 +\! 338a +\! 552\bigr)\lambda^{14} \\
&+\! \bigl(a^4 +\! 30a^3 +\! 277a^2 +\! 1028a +\! 1355\bigr)\lambda^{12} -\! \bigl(4a^4 +\! 80a^3 +\! 572a^2 +\! 1780a +\! 2074\bigr)\lambda^{10}+\! \bigl(6a^4 \\&+\! 100a^3 +\! 628a^2 +\! 1794a +\! 1985\bigr)\lambda^8 -\! \bigl(4a^4 +\! 60a^3 + \!358a^2 +\! 1018a +\! 1148\bigr)\lambda^6+\! \bigl(a^4 +\! 14a^3 \\&+\! 89a^2 +\! 288a +\! 364\bigr)\lambda^4 -\! \bigl(4a^2 +\! 28a +\! 48\bigr)\lambda^2.
\end{align*}
\clearpage

(ii). Below are the matrix \( B_{\Pi}(G_2') \) and its characteristic polynomial \( P(B_{\Pi}(G_2'), \lambda) \).
\begin{figure}[htbp]
    \centering
\includegraphics[width=0.7\linewidth]{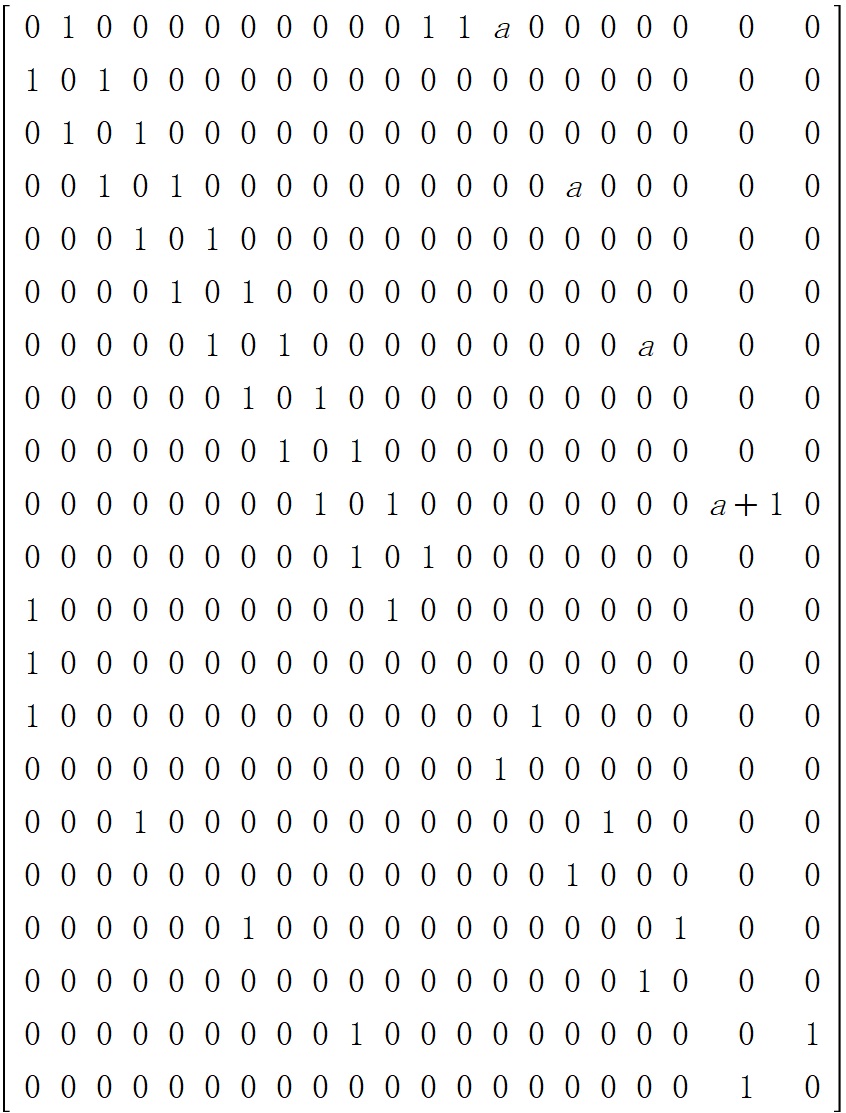}
    \begin{center}
    \small
    The vertices corresponding to the rows/columns of the matrix 
   $B_{\Pi}(G_2')$ are respectively:
    \(
    w_1,\, z_1,\, z_2,\, w_2,\, z_3,\, z_4,\, w_3,\, z_5,\, z_6,\, w_4,\, z_7,\, z_8,\, u,\,u_1,\, v_1,\, u_2,\, v_2,\, u_3,\, v_3,\, u_4,\, v_4.
    \)
  \end{center}
\end{figure}
\begin{align*}
&P(B_\Pi(G_2'), \lambda) =\!\lambda^{21} - \!(18 +\! 4a)\lambda^{19} +\! \bigl(6a^2 +\! 58a +\! 136\bigr)\lambda^{17} - \!\bigl(4a^3 +\! 66a^2 + \!341a +\! 566\bigr)\lambda^{15} \\
&+\! \bigl(a^4 +\! 30a^3 + \!280a^2 +\! 1060a +\! 1434\bigr)\lambda^{13} -\! \bigl(4a^4 +\! 81a^3 + \!594a^2 + \!1911a + \!2309\bigr)\lambda^{11}+ \!\bigl(6a^4\\& +\! 104a^3 +\! 686a^2 +\! 2060a + \!2387\bigr)\lambda^9 -\! \bigl(4a^4 +\! 66a^3 + \!430a^2 +\! 1307a +\! 1552\bigr)\lambda^7 + \!\bigl(a^4 +\! 18a^3 \\&+ \!132a^2 + \!452a +\!595\bigr)\lambda^5 -\! \bigl(a^3 +\! 14a^2 + \!69a +\! 115\bigr)\lambda^3 +\! (2a +\! 7)\lambda.
\end{align*}
Similarly, the characteristic polynomial of $B_\Pi(G_2'')$ is as follows.
\begin{align*}
&P(B_\Pi(G_2''), \lambda) =\lambda^{21} -\! (18 +\! 4a)\lambda^{19} +\! \bigl(6a^2 +\! 58a +\! 136\bigr)\lambda^{17} -\! \bigl(4a^3 +\! 66a^2 +\! 341a +\! 566\bigr)\lambda^{15} \\
&+\! \bigl(a^4 +\! 30a^3 +\! 280a^2 +\! 1060a +\! 1434\bigr)\lambda^{13} -\! \bigl(4a^4 +\! 81a^3 +\! 594a^2 +\! 1911a +\! 2308\bigr)\lambda^{11}+\! \bigl(6a^4\\& +\! 104a^3 +\! 686a^2 +\! 2060a +\! 2382\bigr)\lambda^9 -\! \bigl(4a^4 +\! 66a^3 +\! 430a^2 +\! 1307a +\! 1542\bigr)\lambda^7+\! \bigl(a^4 +\! 18a^3\\& + \!132a^2 + \!452a +\! 585\bigr)\lambda^5 -\! \bigl(a^3 +\! 14a^2 +\! 69a +\! 110\bigr)\lambda^3 +\! (2a +\! 6)\lambda.
\end{align*}

\clearpage

(iii). Below are the matrix \( B_{\Pi}(G_3') \) and its characteristic polynomial \( P(B_{\Pi}(G_3'), \lambda) \).
\begin{figure}[htbp]
    \centering
\includegraphics[width=0.7\linewidth]{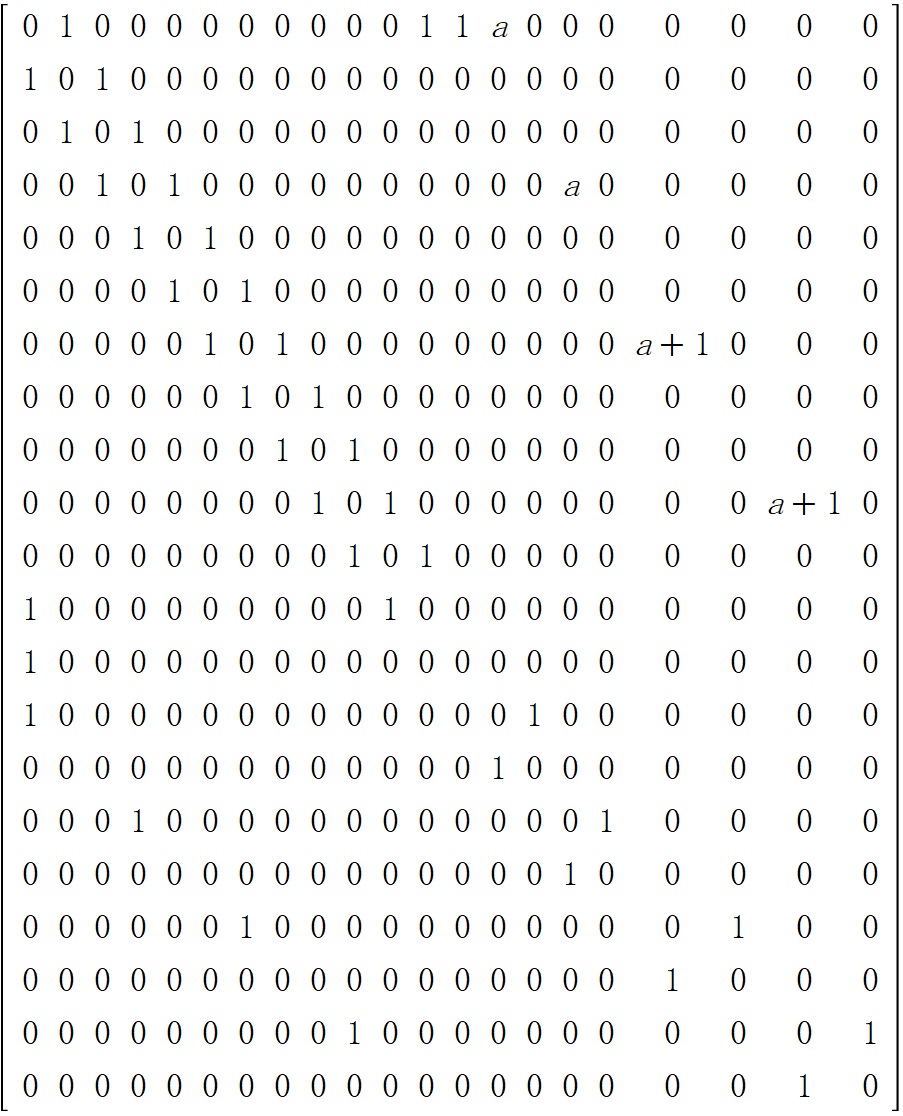}
    \begin{center}
     \small
    The vertices corresponding to the rows/columns of the matrix 
   $B_{\Pi}(G_3')$ are respectively:
    \(
    w_1,\, z_1,\, z_2,\, w_2,\, z_3,\, z_4,\, w_3,\, z_5,\, z_6,\, w_4,\, z_7,\, z_8,\, u,\,u_1,\, v_1,\, u_2,\, v_2,\, u_3,\, v_3,\, u_4,\, v_4.
    \)
  \end{center}
\end{figure}
\begin{align*}
&P(B_\Pi(G_3'), \lambda)=\lambda^{21} -\! (19 +\! 4a)\lambda^{19} +\! \bigl(6a^2 +\! 61a +\! 151\bigr)\lambda^{17} -\! \bigl(4a^3 +\! 69a^2 +\! 375a +\! 657\bigr)\lambda^{15} \\
&+\! \bigl(a^4 +\! 31a^3 +\! 303a^2 +\! 1208a +\! 1725\bigr)\lambda^{13} -\! \bigl(4a^4 +\! 85a^3 +\! 657a^2 +\! 2232a +\! 2847\bigr)\lambda^{11}+\! \bigl(6a^4\\& +\! 110a^3 +\! 768a^2 +\! 2439a +\! 2980\bigr)\lambda^9 -\! \bigl(4a^4 +\! 70a^3 +\! \!483a^2 +\! 1551a + \!1935\bigr)\lambda^7+\! \bigl(a^4 +\! 19a^3 \\&+\! 147a^2 +\! 530a +\! 728\bigr)\lambda^5 -\! \bigl(a^3 +\! 15a^2 +\! 78a +\! 134\bigr)\lambda^3 +\! (2a +\! 7)\lambda.
\end{align*}
Similarly, the characteristic polynomial of $B_\Pi(G_3'')$ is as follows.
\begin{align*}
&P(B_\Pi(G_3''), \lambda)=\lambda^{21} -\! (19 + 4a)\lambda^{19} +\! \bigl(6a^2 + \!61a +\! 151\bigr)\lambda^{17} -\! \bigl(4a^3 + \!69a^2 +\! 375a +\! 657\bigr)\lambda^{15} \\
&+\! \bigl(a^4 +\! 31a^3 +\! 303a^2 + \!1208a +\! 1725\bigr)\lambda^{13} -\! \bigl(4a^4 +\! 85a^3 +\! 657a^2 + \!2232a +\! 2847\bigr)\lambda^{11}+\! \bigl(6a^4\\& + \!110a^3 + \!768a^2 +\! 2439a +\! \!2981\bigr)\lambda^9 -\! \bigl(4a^4 +\! 70a^3 +\! 483a^2 +\! 1551a +\! \!1939\bigr)\lambda^7+\! \bigl(a^4 +\! 19a^3 \\&+\! 147a^2 +\! 530a +\! 734\bigr)\lambda^5 -\! \bigl(a^3 +\! 15a^2 +\! 78a +\! 138\bigr)\lambda^3 +\! (2a +\! 8)\lambda.
\end{align*}

\end{document}